\documentclass[11pt,reqno]{amsart}

\usepackage[utf8]{inputenc}
\usepackage[T1]{fontenc}

\usepackage{geometry}
\usepackage{mathtools}
\usepackage{amssymb}
\usepackage{mathrsfs}
\usepackage{bbm}
\numberwithin{equation}{section}

\theoremstyle{plain}
\newtheorem{theorem}{Theorem}[section]
\newtheorem{proposition}[theorem]{Proposition}
\newtheorem{lemma}[theorem]{Lemma}
\newtheorem{corollary}[theorem]{Corollary}

\theoremstyle{definition}

\theoremstyle{remark}

\usepackage{graphicx}
\usepackage[nopatch=footnote]{microtype}
\usepackage{xcolor}

\usepackage[toc]{appendix}

\usepackage[
  colorlinks=true,
  linkcolor=blue,
  citecolor=blue,
  urlcolor=blue,
  pdftitle={Maximum of the Two-Dimensional Directed Polymer in the Critical Window},
  pdfauthor={Clement Cosco, Shuta Nakajima, Ofer Zeitouni}
]{hyperref}
\usepackage[capitalise,nameinlink]{cleveref}

\title[Maximum of Two-Dimensional Critical  Directed Polymer]{Maximum of the Two-Dimensional Directed Polymer in the Critical Window}
\author{Clement Cosco}
\address[Clement Cosco]{Université Paris Dauphine, Paris, France}
\email{cosco@ceremade.dauphine.fr}
\author{Shuta Nakajima}
\address[Shuta Nakajima]{Department of Mathematics, Keio University, Japan}
\email{njima@keio.jp}
\author{Ofer Zeitouni}
\address[Ofer Zeitouni]{Weizmann Institute, Rehovot, Israel}
\email{ofer.zeitouni@weizmann.ac.il}
\thanks{O.Z. was supported by Israel Science Foundation grant \#621/24.}
\date{}

\newcommand{\E}{\mathbb{E}}
\newcommand{\Pbb}{\mathbb{P}}
\newcommand{\Zbb}{\mathbb{Z}}
\newcommand{\Rbb}{\mathbb{R}}
\newcommand{\1}{\mathbbm{1}}

\newcommand{\F}{\mathcal{F}}

\def\corO{}
\begin{document}

\begin{abstract}
{Let $Z_N^{\beta_N}(x)$ denote the partition function of the two-dimensional directed polymer started at $x\in \mathbb{Z}^2$, in Gaussian environemnt, in the critical window, with parameter $\beta_N=\sqrt{(1-\vartheta/\log N)/R_N}$. Here $R_N\sim (\log N)/\pi$ is the truncated collision Green function at the origin for the simple random walk on $\mathbb{Z}^2$. Let $\mathfrak M_N$ denote the maximum of $\log Z_N(x)$ over $|x|_\infty\leq \sqrt{N}$.
We prove that, for any $\vartheta\in \mathbb{R}$,  in probability, $$\mathfrak M_N/\sqrt{\log N}\to 2\sqrt{2}\text{ as }N\to\infty.$$}
\end{abstract}
\maketitle

\section{Main result}
Let ${S_\cdot}$ {denote}  the nearest-neighbor simple random walk on $\Zbb^2$, and let
$\{\omega(n,x):n\in\mathbb N,\ x\in\Zbb^2\}$ {(\textit{the environment})} be independent standard
Gaussian random variables. We write ${\mathrm E} _x$ and $\mathrm P_x$ for  expectation and probability of the simple random walk 
started at $x$, and $\mathbb{E} $ and $\mathbb P$ for expectation and
probability with respect to the environment. When no confusion occurs, we use the latter also as expectation and probability over both the walk and the environment, i.e. under the \textit{annealed} law.

Fix $\vartheta\in\Rbb$ throughout the paper. For all sufficiently large $N$, set\footnote{\corO{This choice slightly differs from the usual convention where the right hand side of \eqref{eq-091026a} equals $e^{\beta_N^2}-1$. Up to modifying the value of $\vartheta$, this has no effect on the arguments.}}
\begin{equation}
\label{eq-091026a}
\beta_N^2:=\frac{1-\vartheta/\log N}{R_N},
\end{equation}
where
\begin{equation}
\label{eq-091026b}
R_N:=\sum_{n=1}^N\mathrm P_0(S_{2n}=0)
=\frac{\log N}{\pi}+O(1).
\end{equation}
Define the (polymer) partition function 
\begin{equation}
\label{eq-091026c}
Z_N^{\beta_N}(x):={\mathrm E} _x\left[\exp\left(\sum_{n=1}^{N}
\left(\beta_N\omega(n,S_n)-\frac{\beta_N^2}{2}\right)\right)\right].
\end{equation}
We write ${\rm B}(z,r):=z+{\rm B}(r)$ for $z\in\mathbb Z^2$ with ${\rm B}(r):=[-r,r]^2\cap \Zbb^2$. Consider the maximum of the partition functions over the box ${\rm B}(\sqrt{N})$: 
\begin{equation}
\label{eq-091026d}
\mathfrak M_N:=\max_{x\in {\rm B}(\sqrt N)} \log Z_N^{\beta_N}(x).
\end{equation}
{Our main result concerns the asymptotics of $\mathfrak M_N$.}
\begin{theorem}\label{thm:critical-maximum}
As $N\to\infty$,
\[
\frac{\mathfrak M_N}{\sqrt{\log N}}\xrightarrow[N\to\infty]{\mathbb P} 2\sqrt{2}.
\]
\end{theorem}
\subsection{Motivation and previous results}
This work is part of a long sequence of works elucidating the structure of the polymer measure and partition function
in the critical dimension $d=2$. In that dimension, replacing $\beta_N$ by $\beta>0$ independent of $N$  leads to $Z_N^\beta(x)\to 0$ almost surely.
The particular rescaling $\beta_N=\hat\beta/\sqrt{R_N}$ with $\hat\beta\in (0,1)$
was discovered in the context of the stochastic heat equation (SHE)  by Bertini and Cancrini \cite{Bertini98} and was later generalized by Caravenna, Sun, and Zygouras \cite{CaraSuZy-universalityrelev}, in both the SHE and polymer setups, to a wider range of parameters for which a phase transition occurs at $\hat\beta=1$. See also \cite{CaSuZy18,ChDu18,Gu18KPZ2D,NaNa21}. In particular, it follows from these works that for $\hat\beta<1$, the variable $W_N(x):=\log Z_N^{\hat \beta/\sqrt{R_N}}(x)$ converges in distribution to a normal variable of mean
$- {\lambda^2}/ 2$ and variance $\lambda^2$, where
$\lambda^2=\lambda^2(\hat\beta)=-\log (1-{\hat \beta}^2)$. We refer to this regime as the \textit{subcritical} regime.
Still in the subcritical regime and concerning the  spatial behavior of $W_N(x)$,
it was shown in  \cite{CaSuZy18}, that the process
$$G_N(x):=\sqrt{\log N}(W_N(\lfloor x\sqrt{N}\rfloor) -\E W_N(0)), \quad x\in [0,1]^2$$
 converges in law, when tested against smooth and compactly supported functions, to a multiple of a
 log-correlated Gaussian field on $\mathbb R^2$. 
For a related result in the KPZ/SHE setup, see \cite{CaSuZy18,Gu18KPZ2D,NaNa21}.

Given the convergence to a logarithmically correlated field, a natural question is to study the extremes of the field $G_N(x)$. 
We undertook this task in \cite{CNZ-max}, and showed that the maximum, to leading order, can be computed by 
a toy model consisting of a branching random walk with generation-dependent variance. The proof in \cite{CNZ-max} is based 
on high moment computations of  \cite{CN-moments}, which in turn follow the earlier work \cite{cosco2021MomentsUB, cosco2023momentsLB, LiZy}. See also \cite{CCR25}.

Moving to the critical regime, that is $\beta_N$ with the scaling \eqref{eq-091026a}, an important observation of \cite{CaSuZyCrit21}
is that the field $Z_N^\beta(\lfloor x\sqrt{N}\rfloor)$ itself (i.e., with no additional rescaling), diverges pointwise, but converges as a distribution to a \textit{log-log} correlated field. The associated partition function process gives rise to the critical 2d stochastic heat flow, which is not a 
Gaussian multiplicative chaos  \cite{CaSuZyCrit23} and seems of a universality class of its own. This has led to a burst of activity,
see \cite{ZygourasReview,Liu-Zygouras,ClarkMian} for some of the subsequent work.  It is  natural to ask about the maximum in \eqref{eq-091026d}, which is what this paper treats.

\subsection{Notation}

For real $t\geq0$, set
\[
e_t:=\exp\left(\sum_{n=1}^{\lfloor t\rfloor}
\left(\beta_N\omega(n,S_n)-\frac{\beta_N^2}{2}\right)\right),
\qquad
Z_t(x):={\mathrm E} _x[e_t],
\]
and
\[
Z_t(0,y):={\mathrm E} _0[e_t\mathbf 1_{\{S_{\lfloor t\rfloor}=y\}}],\;
Z_t:=Z_t(0)=\sum_{y\in\mathbb Z^2}Z_t(0,y),\;
 {{\mbox{\rm with} \ \mu_t(y):=\frac{Z_t(0,y)}{Z_t(0)}.}}
\]
{Note that these expressions depend on $N$  (and $\vartheta$) through $\beta_N$ but we typically omit both from the notation.} {If we want to emphasize the $\beta$-dependence of $Z_t$, we write $Z_t[\beta_N]$, not to be confused with $Z_t(x)$.}
Set $t_0=0$ and
\begin{equation}
\label{eq-111026e}
{K_N:=\lfloor (\log 2)^{-1}\log \log N \rfloor,\quad t_k:=N^{1-2^{-k}},\quad  \mbox{\rm for}\  1\leq k< K_N, \; \; t_{K_N} := N.}
\end{equation}
The choice of $K_N$ gives $e^{-4}N<t_{K_N -1}\leq e^{-2}N$. For $1\leq k<K_N-1$,
\[
t_{k+1}
/t_k = N^{-2^{-(k+1)}+2^{-k}} = N^{2^{-k-1}}.
\]
Set $\sigma_N^2:=e^{\beta_N^2}-1$.

We write
\[\mu_t^{\otimes 2}(x,y):=\mu_t(x) \mu_t(y) := \frac{Z_t(0,x)Z_t(0,y)}{Z_t(0)^2}.
\]
\subsection{Structure of the paper} We begin by explaining the upper bound proof. The main idea is to write
\begin{equation}
\label{eq-20102026a}
 Z_N(x)=\frac{Z_N(x)}{Z_{N^{1-\varepsilon}}(x)} \cdot Z_{N^{1-\varepsilon}}(x).\end{equation}
By a simple change of coordinates detailed in Section \ref{subsec: subcritical}, the process 
$$\log Z_{N^{1-\varepsilon}}(x)=\log Z_{N^{1-\varepsilon}}^{\beta_N}(\lfloor x\sqrt{N}\rfloor), \quad \lfloor x\sqrt{N}\rfloor\in B(\sqrt{N^{1-\varepsilon}}),$$
is a version of the subcritical polymer studied in \cite{CNZ-max}, and its maximum in that box is to leading order 
$2\sqrt{2} \sqrt{\log N^{1-\varepsilon}}$. Since $B(\sqrt{N})$ can be tiled with a constant multiple of $N^\varepsilon$ such boxes, a slightly more quantitative version of the results in \cite{CNZ-max}, see 
Lemma \ref{lem:CNZ-upper-tail}, allows us to show that with probability tending to $1$ as $N\to\infty$, for any $\eta>0$ and
$\varepsilon$ small enough,
$$\max_{z\in B(\sqrt{N})} \log Z_{N^{1-\varepsilon}}(z)\leq (2\sqrt{2}+\eta)\sqrt{\log N},$$
see Proposition \ref{prop:Subcritical}. 

The problem then reduces for obtaining an upper  bound on the maximum of $\log Z_N(x)-\log Z_{N^{1-\varepsilon}}(x)$ over $x\in (0,1)$,
which is the technically most involved part of the argument. To do that, we decompose 
\[\frac{Z_N(x)}{Z_{N^{1-\varepsilon}}(x)} =\prod_{k=K_0}^N \frac{Z_{t_{k+1}}(x)}{Z_{t_k(x)}},\]
where $K_0=K_0(\varepsilon)$ is large enough.
For a single point $x$, a good control on the lower tail of $\log Z_N(x)-\log Z_{N^{1-\varepsilon}}(x)$  can be obtained by known results, see Lemma \ref{lem: Markov for ZN}. To obtain the maximum over the whole box, we use a chaining argument that involves  controlling differences of $Z_{t_{k+1}}(x)/Z_{t_k}(x)$ at different points $x$. This in turns involves negative moments of $Z_{t_k}(x)$ 
 (see Proposition  \ref{prop:iterated-lower-tail} and also \cite{LX-moments}) and a swapping argument  for controlling moments 
of $Z_t(x)Z_s(y)-Z_s(x)Z_t(y)$, see Proposition \ref{prop:reductionToSubcritical} (proved in Section \ref{sec: difference bound})
and the swapping lemma (Lemma \ref{lem: swapping lemma})
in Appendix \ref{sec-swap}.

The proof of the lower bound is simpler: by the change of coordinates discussed above, we can find a point $x_*\in B(\sqrt{N})$
so that, for any fixed $\eta>0$ and $\varepsilon$ small enough, $\log Z_{N^{1-\varepsilon}}(x_*)>(2\sqrt{2}-\eta)\sqrt{\log N}$, and further this $x^*$ is measurable on $\{\omega(n,x), n\leq N^{1-\varepsilon}, x\in \mathbb Z^2\}$. We can then use our negative moment bounds 
and the Markov property to control the left tail of $\log Z_N(x_*)-\log Z_{N^{1-\varepsilon}}(x_*)$. Together, these completes the proof. Details are in Section \ref{sec-LB}.

\subsection{Conventions}
We use symbols $c,C$ for small and large constants that may change from line to  line. If they depend on parameters, we may explicitly write those. We use the Landau notation: if quantities $A_N,B_N$ satisfy that $\limsup_{N\to\infty} |A_N/B_N|<\infty$, we write $A_N=O(B_N)$. If the limsup vanishes, we write $A_N=o(B_N)$. If the parameter on which the quantities depend is not $N$, but e.g. 
$\varepsilon$, we similarly write $A_\varepsilon=O(B_\varepsilon)$. We recall that ${\rm B}(r):=[-r,r]^2\cap \Zbb^2$ and, 
for $z\in\mathbb Z^2$,  ${\rm B}(z,r):=z+{\rm B}(r)$. For points $z=(z_1,z_2)\in \mathbb{Z}^2$, we write $\lfloor z\rfloor=
(\lfloor z_1\rfloor, \lfloor z_2\rfloor)$.
\section{Notation and Preliminaries}
\label{sec-NotPre}
We collect in this section  preliminary estimates.

\subsection{Positive moment bounds}
We record moment estimates for the ratio of partition functions from  previous work. {We rely heavily on \cite{CN-moments}, whose terminology we adopt.}
We will often use the following change of variables in time and inverse temperature:   for \(t\le N\), we define
\begin{equation}\label{eq: transformation of Zt}
  \widetilde{\beta}_t^2 := \frac{R_t}{R_N}(1-\vartheta/\log N).
\end{equation}
so that $\widetilde{\beta}_t^2/R_t = \beta_N^2$.
\corO{With this notation,
\begin{equation}
\label{eq-301026a}{Z_t[\beta_N]=Z_t[\widetilde{\beta}_t/\sqrt{R_t}]}.
\end{equation} 
In particular, if $t\ll N^{1-\varepsilon}$ for some $\varepsilon>0$ then $\widetilde{\beta}_t<1-\varepsilon/4$ and 
one is in the subcritical regime we introduced above  (with $t$ replacing $N$). Similarly, if $t=t_N$ is such that  $\log t/\log N\to 1$ but $t\ll N$ then one is in the subcritical regime (again, replacing $t$ by $N$ and using $\widetilde \beta_t$).}

\corO{In any case, we can apply \cite[Proposition~A.2]{CSZ-Dickman} and \cite[Proposition~A.2~(i)]{CN-moments} to obtain that}
 there exists $C>0$ such that
\begin{equation} \label{eq:2ndMomentBoundIntroduction}
\mathbb{E} [Z_{t_k}^2]\leq C2^k,\quad \forall k<K_N,
\end{equation}
Ratios can also be controlled, as follows.
\begin{lemma}\label{lem: Markov for ZN}
Let $p\geq 1$. For non-negative integers $s,r$, 
$$ \mathbb{E}  \left[  \left(\frac{Z_{s+r}}{Z_{r}}\right)^p \right] \leq \mathbb{E}  \left[  Z_{s}^p \right],\quad \mathbb{E} [Z_{r+s}^p] \leq  \mathbb{E} [Z_r^p]\mathbb{E} [Z_s^p].$$
\end{lemma}
\begin{proof}
Define the filtration $\F_m:=\sigma(\omega(j,z):1\le j\le m,\ z\in\Zbb^2)$.     Write $Z_{r,s+r}((r,z))$ for the partition function over
$(r,s+r]$ started from $z$ at time $r$.  By the Markov property,
$$
\frac{Z_{s+r}}{Z_{r}}
=\sum_{z\in\mathbb Z^2}\mu_{r}(z)Z_{r,s+r}((r,z)).
$$
Conditional Jensen's inequality and translation invariance give
\begin{equation}
\label{eq-111026a}
\mathbb{E} \left[
\left.\left(\frac{Z_{s+r}}{Z_{r}}\right)^p
\right|\mathcal F_{r}
\right]
\leq\sum_z\mu_{r}(z)\,
\mathbb{E} [Z_{r,s+r}((r,z))^p]
\leq \mathbb{E} [Z_{s}^p].
\end{equation}
Taking expectations yields
\begin{equation}\label{eq:ratio-by-Ap}
\mathbb{E} \left[
\left(\frac{Z_{s+r}}{Z_{r}}\right)^p
\right]\leq \mathbb{E} [Z_{s}^p].
\end{equation}
{Further,
 for nonnegative integers $r,s$, we have} 
$$
\mathbb{E} [Z_{r+s}^p] =\mathbb{E} [Z_{r}^p \mathbb{E} [(Z_{r+s}/Z_r)^p|\mathcal F_r]] \leq  \mathbb{E} [Z_r^p]\mathbb{E} [Z_s^p],
$$
{where the  inequality follows from \eqref{eq-111026a}.}
\end{proof}
\begin{lemma}
\label{lem:moment-all-scales}
Let $T_N$ and $M_N$ be sequences of positive integers such that $\sqrt N\leq T_N\leq N$, $M_N\leq \log N$ and  $NM_N/T_N$ diverges as $N\to \infty$. For any $\delta>0$, there exists a constant $C=C(\delta)>0$ such that the following holds:  if integers $\sqrt{N}\leq T\leq T_N$ and
$p\geq 2$  satisfy   
\begin{equation}\label{eq:admissible-p-M}
    \binom{p}{2} \leq  (1-\delta) (\log (N/T)+\log M_N),
\end{equation}
 then
\begin{equation}\label{eq:block-moment-ratio}
\mathbb{E} \left[
Z_T^p
\right]
\leq  \exp\left(
CM_N\binom p2\log\left(\frac{4\log N}{\log (N/T)+\log M_N}\right)
\right).
\end{equation}
\end{lemma}

\begin{proof}
We write $M=M_N$ for brevity. 
 {For $p=2$, the claim follows from \eqref{eq:2ndMomentBoundIntroduction}  and the second part of Lemma \ref{lem: Markov for ZN}.  Therefore in what follows we consider $p\geq 3$ integer, and again
appeal to Lemma~\ref{lem: Markov for ZN}.} Splitting the interval of length $T$ into $M$ intervals of length
at most $\lceil T/M\rceil$, applying the {second part of the lemma} $M$ times iteratively, {and using that the lemma implies also  that} the map $t\mapsto \mathbb{E} [Z_t^p]$ is non-decreasing, we obtain
\begin{align}\label{eq: submodularity}
    \mathbb{E} [Z_{T}^p]
\leq \left(\mathbb{E} [Z_{\lceil T/M\rceil}^p]\right)^{M}.
\end{align}

It remains to estimate $\mathbb{E} [Z_{\lceil T/M\rceil}^p]$. Put
$n:=\lceil T/M\rceil$. Since $T/M$ is sufficiently large due to $T\geq \sqrt N$ and $M\leq \log N$,
$$
\log n=\log\left\lceil\frac{T}{M}\right\rceil
=\log N-\log (N/T)-\log M+O(1).
$$
Recall that $\sigma_N^2=e^{\beta_N^2}-1 = \pi/\log N + O(1/(\log N)^2)$. We note that
$$
b_n:= \left(1-\frac{\sigma_N^2}{\pi}\log n\right)^{-1}
=\frac{\log N}{\log (N/T)+\log M+O(1)}.
$$
Since $\log (N/T) + \log M$ diverges to $\infty$ uniformly in $T\leq T_N $ by the assumption $N M_N/T_N \to \infty$, we have
$$
b_n\frac{\sigma_N^2}{\pi}= \frac{1+o(1)}{\log (N/T)+\log M},
$$
 uniformly in  $\sqrt{N}\leq T\leq T_N $.
Hence, the  condition  
\eqref{eq:admissible-p-M} implies that for $N$ sufficiently large, 
\begin{align}\label{eq: cond for bn}
    \binom p2 b_n\frac{\sigma_N^2}{\pi}\leq (1-\delta/2).
\end{align}

Since $1\ll n = \lceil T/M\rceil \ll N$ uniformly in  $\sqrt N\leq T< T_N$ due to $M\leq \log N$, and $N M_N/T_N\to \infty$,  we have
$$
\left(1-\sigma_N^2 R_n\right)\log n
=\log n -\frac{(\log n)^2}{\log N}+O(1),
$$
which diverges uniformly in $\sqrt{N}\leq T\leq T_N $ as $N\to \infty$. Hence, $\beta_N$ is in the subcritical  or quasicritical  regime at time $n$  (see the remark after \eqref{eq: transformation of Zt}). \corO{Using the same reparametrization of $\beta_N$ as discussed after \eqref{eq-301026a}, with $t$ replaced by $n$, we have from  \cite[Theorem~2.7]{CN-moments} with $s=1$ and $\alpha=1-\delta/2$, 
that}
$$
\mathbb{E} [Z_n^p]
\leq Cp^C b_n^{C\binom p2}
\leq\exp\left(C\binom p2(1+\log b_n)\right). 
$$
By 
$
\log (N/T)+\log M\leq 2\log N
$  due to $M\leq \log N$,
we have, for $N$ sufficiently large,
$$
1+\log b_n
\leq C\log\left(\frac{4\log N}{\log (N/T)+\log M}\right).
$$
Together with \eqref{eq: submodularity}, this proves \eqref{eq:block-moment-ratio}.
\end{proof}
The following bound is not optimal but useful for later arguments. Note that a sharper bound is obtained in \cite{Liu-Zygouras} for the continuum model. 
\begin{corollary}\label{cor:positive-moment-bound}
For every  $p\geq 2$, there exists $C=C(p)>0$ such that, for all
sufficiently large $N$,
\[
\sup_{T\leq N}\mathbb{E} [Z_T^p]\leq \exp\left(C(\log\log N)^2\right).
\]
\end{corollary}
\begin{proof}
{By Jensen's inequality, it suffices to prove the result for $p$ integer.}
{Further, it suffices to remove the supremum and only consider  $T=N$ because, by Lemma  \ref{lem: Markov for ZN},}
the moments increase in   $T$. Set
$M:=M_N := \lceil\log\log N\rceil$. 
Then,  the  condition  
\eqref{eq:admissible-p-M} is satisfied because $p$ is fixed and $M$ goes to infinity. 
 By \eqref{eq:block-moment-ratio}, we have 
 
$$
\mathbb{E} [Z_N^p]
\leq\exp(C(\log\log N)^2),
$$
with some $C=C(p)$. 
\end{proof}
\subsection{Lower tail concentration and negative moments}
\label{sec-negmom}

The following proposition was proved in \cite{LX-moments}. For completeness, we give a proof using Chen’s convexity theorem \cite[Theorem 1]{Chen}. 
\begin{proposition}\label{prop:iterated-lower-tail}
Let $T\in\mathbb N$ and let $U>1$ be deterministic. Assume that
\begin{equation}\label{eq:second-moment}
 \mathbb{E} [Z_T^2]\le U.
\end{equation}
Then, for any $p>0$, 
\begin{equation}\label{eq:negative-moments}
 \mathbb{E}  Z_T^{-p}
 \le\exp\left(\frac{p(p+1)}2 {\log U}\right).
\end{equation}
\end{proposition}

\begin{proof}
For fixed $T$, the partition function depends on finitely many environment
coordinates, and
$$
 \log Z_T
 =\log\left(4^{-T}\sum_{\gamma}
 \exp\left(\beta_N\sum_{n=1}^T\omega(n,\gamma_n)
 -\frac{T\beta_N^2}{2}\right)\right),
$$
where the sum runs over all paths $\gamma=(0,\gamma_1,\ldots,\gamma_T)$ of length $T$.
 It is well-known that $\omega\mapsto \log Z_T$ is convex due to H\"older's  inequality.

By Chen's Gaussian convexity theorem \cite[Theorem 1]{Chen}, the function
$$
 \Psi(q):=
 \begin{cases}
 q^{-1}\log\mathbb{E}  Z_T^q,&q\ne0,\\
 \mathbb{E} \log Z_T,&q=0,
 \end{cases}
$$
is convex on $\mathbb R$. By \eqref{eq:second-moment},
$$
 \Psi(1)=0,\qquad
 \Psi(2)=\frac12\log\mathbb{E}  Z_T^2\le\frac {\log U}2.
$$
Thus, for every $p>0$,
$$
 \frac{\Psi(1)-\Psi(-p)}{1+p}
 \le\Psi(2)-\Psi(1)\le\frac {{\log U}}2,
$$
and consequently we have  \eqref{eq:negative-moments}.
\end{proof}

\section{Proof of the upper bound in {Theorem \ref{thm:critical-maximum}}}
We prove the upper and lower bounds separately. In either case, the heart of the proof is to reduce attention from $Z_N$ to $Z_{N^{1-\varepsilon}}$, for which the methods of \cite{CNZ-max} apply. For the upper bound, the reduction requires certain
continuity estimates, which are developed in Section~\ref{sec: difference bound}, as well as the a priori positive and negative moment bounds from Section \ref{sec-NotPre}. The proof of the lower bound is much easier, and 
essentially combines the lower bound of \cite{CNZ-max} with lower tail estimates based on the estimates in Section 
\ref{sec-negmom}.

The rest of this section is devoted to the proof of the upper bound. The lower bound will be shown in Section \ref{sec-LB}.

\subsection{Proof of upper bound in Theorem~\ref{thm:critical-maximum}}
The upper bound follows from the next two propositions. The first proposition deals with  {the critical and quasicritical  regimes.
Its {proof} is provided   in the next subsection, and uses estimates from Section~\ref{sec: difference bound}. }
\begin{proposition} \label{prop:reductionToSubcritical}
For any $\delta>0$, if we take $K_0>1$  sufficiently large, then
\[
\limsup_{N\to\infty}\mathbb{P}\left(\exists x\in {\rm B}(\sqrt{N}), \log Z_N(x) - \log Z_{t_{K_0}}(x)>\delta \sqrt{\log N}\right) =0.
\]
\end{proposition}

{The second proposition establishes the necessary controls for the subcritical regime, and   follows from estimates in \cite[Section~3]{CNZ-max}. For completeness, the proof is provided in  Section~\ref{subsec: subcritical}.

\begin{proposition} \label{prop:Subcritical}
For any $\delta>0$, 
\[
\limsup_{\varepsilon\to 0}\limsup_{N\to\infty}\mathbb{P}\left(\exists x\in {\rm B}(\sqrt{N}), \log Z_{N^{1-\varepsilon}}(x)>(2\sqrt{2}+\delta) \sqrt{\log N}\right) =0.
\]
\end{proposition}

Combining {both propositions yields the proof of the upper bound, as follows.}
\begin{proof}[Proof of upper bound in Theorem~\ref{thm:critical-maximum}]
    Let $\delta,\eta>0$ be arbitrary. Then, we take $K_0'\in \mathbb N$ such that for any fixed $K_0>K_0'$,  Proposition~\ref{prop:reductionToSubcritical} holds. By Proposition~\ref{prop:Subcritical}, we take $K>K_0'$ such that if $N$ is sufficiently large, then  
    $$\mathbb{P}\left(\exists x\in {\rm B}(\sqrt{N}), \log Z_{t_K}(x)>(2\sqrt{2}+\delta) \sqrt{\log N}\right)<\eta.$$
    Note that 
    \begin{align}
\label{eq-111026d}
&        \mathbb{P}\left(\exists x\in {\rm B}(\sqrt{N}), \log Z_{N}(x)\geq (2\sqrt{2}+2\delta) \sqrt{\log N}\right)\\
&\leq \mathbb{P}\left(\exists x\in {\rm B}(\sqrt{N}), \log Z_{t_K}(x)>(2\sqrt{2}+\delta) \sqrt{\log N}\right)\nonumber\\
&\qquad + \mathbb{P}\left(\exists x\in {\rm B}(\sqrt{N}), \log Z_{N}(x) -\log Z_{t_K}(x)> \delta \sqrt{\log N}\right).\nonumber
    \end{align}
{By taking $N$ sufficiently large, both terms on right-hand side of \eqref{eq-111026d} are less than $\eta$. Therefore,
since $\eta$ is an arbitrary positive constant, we conclude that
\[ \lim_{N\to\infty} \mathbb{P}\left(\frac{\mathfrak M_N}{\sqrt{\log N}}\geq  2\sqrt{2}+2\delta\right)=0,\]
which completes the proof of the upper bound since $\delta>0$ is arbitrary.}
\end{proof}

\subsection{Chaining in the quasicritical regime - {Proof of Proposition~\ref{prop:reductionToSubcritical}}}\label{subsec: chaining quasicritical}
{The  following proposition is proved in  Section~\ref{sec: difference bound}. Recall the notation $K_N$ and $t_k$, see
\eqref{eq-111026e}.}
\begin{proposition}\label{prop: prop moment bound}
For any $A>1$, there exist $p=p(A)>1$ and $C=C(A,p)>1$ such that for all sufficiently
large $N$, the following holds for $2\leq k<K_N$ and
$|x-y| \leq \sqrt{t_{k}/2}$ with $x-y\in 4\mathbb Z^2$,
\begin{align} \label{eq:8times}
\mathbb{E} \left[\left|\frac{Z_{t_{k+1}}(x)}{Z_{t_{k}}(x)} - \frac{Z_{t_{k+1}}(y)}{Z_{t_{k}}(y)}\right|^p\right]\leq    e^{C k+ C(\log \log N)^2} \left(\frac{|x-y|}{\sqrt{t_k}}\right)^{A}.
\end{align}
\end{proposition}
{The rest of this subsection is devoted to  the proof of Proposition~\ref{prop:reductionToSubcritical}. The proof
uses Proposition \ref{prop: prop moment bound} with $A=4$.}
\begin{proof}[Proof of Proposition~\ref{prop:reductionToSubcritical}]
 If $2^{-k}\log N\geq (\log{\log N})^2$, then \textcolor{blue}{by Lemma \ref{lem: Markov for ZN} and} Lemma~\ref{lem:moment-all-scales} with $M=1$ \textcolor{blue}{and $T=t_{k+1}$}, we have
$$
\mathbb{E} \left[\left(\frac{Z_{t_{k+1}}(x)}{Z_{t_k}(x)}\right)^q\right]
\leq \exp(Cq^2(k+1)),
\qquad 3\leq q\leq c\sqrt{2^{-k}\log N},\quad q\in\mathbb N.
$$
 For $t_k\leq T\leq t_{k+1}$ with $2^{-k}\log N\geq (\log{\log N})^2$,
 we have the bound $\mathbb{E} [Z_T^2]\leq C2^{k+1}$.
Thus, Lemma~\ref{lem:moment-all-scales} with $M=1$, \eqref{eq:negative-moments}, and H\"older's inequality imply, for any $3\leq q\leq \sqrt{2^{-k}\log N}/8$,
$$
\mathbb{E} \left[\left(\frac{Z_{t_k}(x)}{Z_T(x)}\right)^q\right]
\leq \mathbb{E} [Z_{t_k}^{2q}]^{1/2}
       \mathbb{E} [Z_T^{-2q}]^{1/2}
\leq \exp(Cq^2(k+1)).
$$

For the remaining values of $k$, the moment estimate Corollary~\ref{cor:positive-moment-bound} with Lemma \ref{lem: Markov for ZN} gives
$$
\mathbb{E} \left[\left(\frac{Z_{t_{k+1}}(x)}{Z_{t_k}(x)}\right)^3\right]
\leq \exp(C(\log{\log N})^2).
$$

Fix $A=4$, and let $p=p(4)>1$  be as in
Proposition~\ref{prop: prop moment bound}. For $2\leq k<K_N$, set
$$
\begin{aligned}
 r_k:=\frac{1}{36}\sqrt{t_k}e^{-2^{-k}\log N}.
\end{aligned}
$$ Fix $K_0 \geq2$ large enough that
$2^{-K_0 }<1/4$. Then $r_k\to\infty$ uniformly for
$K_0 \leq k<K_N$.
\begin{figure}\label{fig: fig1}
  \centering
  \includegraphics[width=0.85\textwidth]{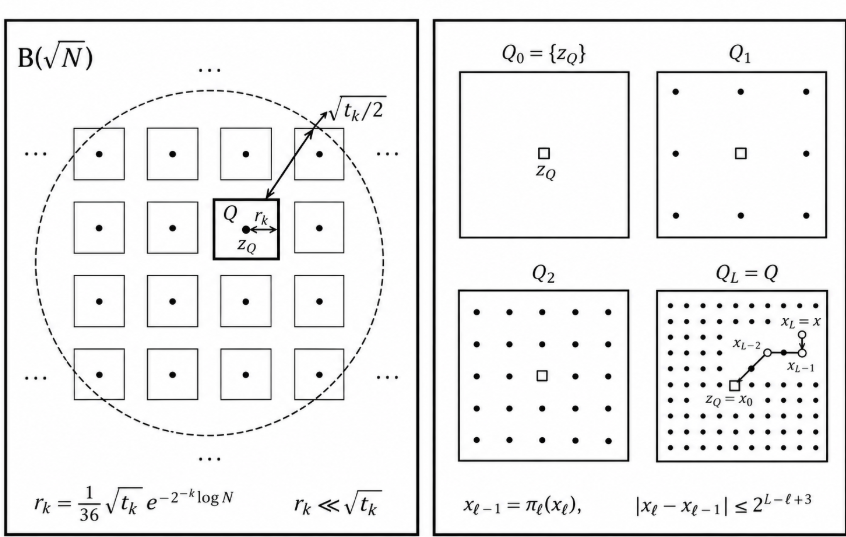}
  \caption{Schematic picture of the covering of ${\rm B}(\sqrt N)$ by sets $Q$ of radius $r_k$ (left), and of the nested nets $Q_\ell$ and the chain from $x\in Q$ to $z_Q$ (right).}
\end{figure}

{We now present a chaining argument. Consider} a covering of ${\rm B}(\sqrt N)$ with balls of radius  $r_k$: we take a family of points in $\mathbb Z^2$, $(z_Q)_{Q\in \mathcal Q}$, such that
$${\rm B}(\sqrt N)\cap \mathbb Z^2 \subset \bigcup_{Q\in \mathcal Q} ({\rm B}(z_Q,r_k)\cap  (z_Q+4\mathbb Z^2)),\quad |\mathcal Q|\leq C\frac{N}{r_k^2}.$$

With abuse of notation, we write $Q$ for ${\rm B}(z_Q,r_k)\cap (z_Q+4\mathbb Z^2)$.  Let $L:=\lfloor \log_2 r_k\rfloor$, so that $2^L \leq r_k<2^{L+1} $.
For $0\leq \ell\leq L$, consider the nets
$$
Q_\ell:=\big(z_Q+2^{L -\ell+2}\,\mathbb Z^2\big)\cap Q .
$$
They increase with $\ell$. Since $2^{L+2} >r_k$, the coarsest net is $Q_0=\{z_Q\}$, and the finest net is $Q_L=Q$.
To each $x\in Q_\ell$ with $\ell\geq 1$, we assign a parent $\pi_\ell(x)\in Q_{\ell-1}$
 so that 
$$
|x-\pi_\ell(x)|\leq 2^{L -\ell+3},
\qquad x-\pi_\ell(x)\in 4\mathbb Z^2 .
$$

We write
$$
\Phi_k(x):=\frac{Z_{t_{k+1}}(x)}{Z_{t_k}(x)} .
$$
Fix $x\in Q$ and follow its ancestors: set $x_L:=x$ and $x_{\ell-1}:=\pi_\ell(x_\ell)$ for $1\leq \ell\leq L$,
so that $x_\ell\in Q_\ell$ and $x_0=z_Q$. Then
$$
\Phi_k(x)-\Phi_k(z_Q)=\sum_{\ell=1}^{L}\big(\Phi_k(x_\ell)-\Phi_k(x_{\ell-1})\big).
$$
 Taking the supremum over $x\in Q$ gives
$$
\sup_{x\in Q}|\Phi_k(x)-\Phi_k(z_Q)|
\leq\sum_{\ell=1}^{L}\max_{u\in Q_\ell}|\Phi_k(u)-\Phi_k(\pi_\ell(u))|.
$$

For each $u\in Q_\ell$, by Proposition~\ref{prop: prop moment bound}, we have 
$$
\mathbb{E} [|\Phi_k(u)-\Phi_k(\pi_\ell(u))|^p]^{1/p}
\leq e^{C(\log\log N)^2}
\left(\frac{r_k}{\sqrt{t_k}}\right)^{4/p}2^{-4\ell/p},
$$
with some $C=C(p)>0$, where we have used $k<K_N=O(\log\log N)$ to absorb the factor
$e^{Ck/p}$. Note that the choice of $r_k$ ensures the condition \corO{needed for the application of
}
 \eqref{eq:8times}.
By Minkowski's inequality, we obtain
$$
\begin{aligned}
\mathbb{E} \left[\sup_{x\in Q}|\Phi_k(x)-\Phi_k(z_Q)|^p\right]^{1/p}
&\leq\sum_{\ell=1}^{L}
\mathbb{E} \left[\max_{u\in Q_\ell}|\Phi_k(u)-\Phi_k(\pi_\ell(u))|^p\right]^{1/p}\\
&\leq\sum_{\ell=1}^{L}
\left(\sum_{u\in Q_\ell}\mathbb{E} [|\Phi_k(u)-\Phi_k(\pi_\ell(u))|^p]\right)^{1/p}.
\end{aligned}
$$
 Since $Q_\ell$ consists of points with spacing $2^{L -\ell+2}$
in a ball of radius $r_k<2^{L+1}$, we have $|Q_\ell|\leq C2^{2\ell}$.
 Hence, we obtain
$$
\mathbb{E} \left[\sup_{x\in Q}|\Phi_k(x)-\Phi_k(z_Q)|^p\right]^{1/p}
\leq e^{C(\log{\log N})^2}\left(\frac{r_k}{\sqrt{t_k}}\right)^{4/p}
\sum_{\ell\geq1}2^{-2\ell/p}.$$
 The series converges, and hence
$$
\mathbb{E} \left[\sup_{x\in Q}|\Phi_k(x)-\Phi_k(z_Q)|^p\right]
\leq C \exp(C(\log{\log N})^2- 2^{2-k}\log N).
$$
We take a constant $B>0$ {large, 
which will be chosen below, and set}
$$b_k:=\sqrt{(k+1)2^{-k}\log N}+(\log{\log N})^3.$$
The union bound with $|\mathcal Q|\leq C e^{3( 2^{-k} \log N)}$ and Markov's inequality now give
\begin{equation}\label{eq:quasicritical-oscillation}
  \mathbb P\left(\exists Q\in \mathcal Q:
\sup_{x\in Q}|\Phi_k(x)-\Phi_k(z_Q)|>\tfrac12e^{B b_k}\right)
\leq C\exp(-2^{-k}\log N-B p\, b_k+C(\log{\log N})^2).
\end{equation}

We set 
$$q:=q_k: = 
\begin{cases}
    \left\lfloor \sqrt{\frac{2^{-k}\log N}{16k}}\right\rfloor,\quad & \text{if }
2^{-k}\log N\geq (\log{\log N})^2\\
 3,&\text{otherwise.}
\end{cases}
$$

When $2^{-k}\log N\geq (\log{\log N})^2$, for all sufficiently large $N$ we have
$q_k\geq3$, $\binom{q_k}{2}\leq 2^{-k}\log N/8$,
so that \eqref{eq:admissible-p-M} holds with $p=q_k$ $M=1$ and $T=t_{k+1}$,  and
Lemmas~\ref{lem: Markov for ZN} and \ref{lem:moment-all-scales} apply. Thus, by the union bound and  invariance by translation,
\begin{align}
\label{eq-111026h}
\mathbb P\left(\exists Q\in \mathcal Q: \Phi_k(z_Q)>\tfrac12e^{B b_k}\right)
& \leq |\mathcal Q|\mathbb P\left(\Phi_k(0)>\tfrac12e^{B b_k}\right)\\
&\leq C\exp(3\cdot 2^{-k}\log N-B q_k b_k+q_k\log2+Cq_k^2(k+1)).\nonumber
\end{align}
Taking $B$ sufficiently large (depending on the constant $C$ above) so that $B q_k b_k$ dominates the others in the exponent,  the left-hand side of \eqref{eq-111026h}  is bounded by 
$$C\exp(-(\log{\log N})^3).$$ 

If $2^{-k}\log N<(\log{\log N})^2$, by Lemma~\ref{lem: Markov for ZN} and Corollary~\ref{cor:positive-moment-bound} with   $p=3$, the left hand side of \eqref{eq-111026h}  is bounded by
 \begin{align*}
     |\mathcal Q|\mathbb P\left(\Phi_k(0)>\tfrac12e^{B b_k}\right)&\le |\mathcal Q| e^{- 3 B b_k} \mathbb{E} [(Z_{\lfloor t_{k+1}\rfloor -\lfloor t_k\rfloor })^3]\\
     & \leq C\exp(C(\log{\log N})^2-2(\log{\log N})^3).
      \end{align*}

Note that  $\Phi_k(x)>e^{B b_k}$ for some $x\in Q$ implies either $\Phi_k(z_Q)>\tfrac12e^{B b_k}$ or $|\Phi_k(x)-\Phi_k(z_Q)|>\tfrac12e^{B b_k}$. Together with the  bound \eqref{eq:quasicritical-oscillation}, the union bound yields that for $K_0 \leq k<K_N$,
$$
\mathbb P\left(\sup_{x\in {\rm B}(\sqrt N)} \Phi_k(x)>e^{B b_k}\right)
\leq Ce^{-(\log{\log N})^3},
$$
for all sufficiently large $N$. Since $K_N\leq C \log\log N$,
\begin{equation}
\label{eq-121026b}\mathbb P\left(\exists k\in [K_0,K_N)\cap \mathbb Z,\,\sup_{x\in {\rm B}(\sqrt N)}\log\Phi_k(x)>B b_k\right) \leq C K_N e^{-(\log{\log N})^3},\end{equation}
which converges to $0$.

On the complement of the event {inside the probability in \eqref{eq-121026b},} for every $x\in {\rm B}(\sqrt N)$,
$$
\begin{aligned}
\log\frac{Z_N(x)}{Z_{t_{K_0}}(x)}
&=\sum_{k=K_0}^{K_N-1}\log\Phi_k(x)
\leq\sum_{k=K_0}^{K_N-1}B b_k
\leq C\sqrt{(K_0+1)2^{-K_0}\log N}+C((\log\log N)^4),
\end{aligned}
$$
with some $C$ depending on $B$. Taking $K_0$ large enough so  that $C\sqrt{(K_0+1)2^{-K_0}}<\delta/2$
and then $N$ sufficiently large, this is bounded by $\delta \sqrt{\log N}$. {This completes the proof of Proposition~\ref{prop:reductionToSubcritical}.}
\end{proof}
\subsection{The subcritical estimate}
\label{subsec: subcritical}

We prove here Proposition~\ref{prop:Subcritical}.  {The proof is a rather straight-forward consequence of the following  quantitative
upper bound from \cite[Section 3]{CNZ-max}, rewritten in our notation.}
\begin{lemma}
\label{lem:CNZ-upper-tail}
For every $\eta>0$ there exist $c>0$ and
$\varepsilon_0>0$ such that the following holds.
Let $0<\varepsilon<\varepsilon_0$, and let $\beta_N\geq0$ be a
deterministic sequence satisfying
$$
    \widehat\beta_{N,\varepsilon}^2
  :=\beta_N^2R_{\lfloor N^{1-\varepsilon}\rfloor}
    =1-\varepsilon+O((\log N)^{-1}),
$$
where
$$
    R_n:=\sum_{j=1}^n\mathrm P_0(S_{2j}=0)
       =\frac{\log n}{\pi}+O(1).
$$
Then, uniformly in $z\in\mathbb Z^2$, for all sufficiently large $N$
depending on $\eta,\varepsilon$ and the error bound in the hypothesis,
\begin{equation}
\label{eq:CNZ-tail}
 \mathbb P\left(
   \max_{x\in {\rm B}(z,\sqrt{N^{1-\varepsilon}})}
    \log Z_{N^{1-\varepsilon}}^{\beta_N}(x)
     >(2\sqrt2+\eta)\sqrt{\log N}
 \right)
 \le N^{-c\sqrt{\varepsilon}} .
\end{equation}
\end{lemma}
{The proof of Lemma \ref{lem:CNZ-upper-tail} is based on \cite[Section 3]{CNZ-max}, but requires following the constants there 
to account for the dependence on $\varepsilon$, which measures the distance to critically. For completeness, we provide the proof in Appendix
\ref{app-pflem3.4}.}
\begin{proof}[Proof of Proposition~\ref{prop:Subcritical}]
 Fix $\delta>0$.  Let $\varepsilon>0$ be arbitrary.  Put 
\[
\widehat\beta_{N,\varepsilon}^2
:=\beta_N^2R_{\lfloor N^{1-\varepsilon}\rfloor}.
\]
By the definition of $\beta_N$,
\[
\begin{aligned}
 \widehat\beta_{N,\varepsilon}^2
 &=
 \left(1-\frac{\vartheta}{\log N}\right)
 \frac{(1-\varepsilon)\log N+O(1)}{\log N+O(1)}
=
 1-\varepsilon+O((\log N)^{-1}).
\end{aligned}
\]
Thus Lemma~\ref{lem:CNZ-upper-tail} applies.

We next cover the box ${\rm B}(\sqrt N)$ by translates of
boxes of radius $\sqrt{N^{1-\varepsilon}}$.  There exist
$z_1,\ldots,z_{J_N}\in\mathbb Z^2$ such that
\[
  {\rm B}(\sqrt N)
  \subset
  \bigcup_{j=1}^{J_N} {\rm B}(z_j,\sqrt{N^{1-\varepsilon}})
\]
and
\[
  J_N
  \le
  C\left(1+\frac{\sqrt N}{\sqrt{N^{1-\varepsilon}}}\right)^2
  \le C N^\varepsilon .
\]

Apply Lemma~\ref{lem:CNZ-upper-tail}, with
$\eta:=\delta/2$.  Since
\(
  2\sqrt2+\frac\delta2
  <2\sqrt2+\delta,
\)
we obtain, by the union bound and spatial translation invariance,
\begin{align*}
&\mathbb P\left(
   \exists x\in {\rm B}(\sqrt N):
  \log Z_{N^{1-\varepsilon}}^{\beta_N}(x)
  >(2\sqrt2+\delta)\sqrt{\log N}
   \right)\\
&\qquad\le
 \sum_{j=1}^{J_N}
 \mathbb P\left(
  \max_{x\in {\rm B}(z_j,\sqrt{N^{1-\varepsilon}})}
    \log Z_{N^{1-\varepsilon}}^{\beta_N}(x)
  >(2\sqrt2+\delta)\sqrt{\log N}
 \right)\\
&\qquad\le
 C N^\varepsilon
 N^{-c_\delta\sqrt\varepsilon}=C N^{\,\varepsilon-c_\delta\sqrt\varepsilon},
\end{align*}
with some $c_\delta>0$ depending on $\delta$. Choose $\varepsilon_0:=\varepsilon_0(\delta)>0$ so small that for $0<\varepsilon<\varepsilon_0$,
\[
  \varepsilon-c_\delta\sqrt\varepsilon
  \le -\frac{c_\delta}{2}\sqrt\varepsilon.
\]
Consequently, for every fixed
$0<\varepsilon<\varepsilon_0$,
\[
 \limsup_{N\to\infty}
 \mathbb P\left(
   \exists x\in {\rm B}(\sqrt N):
   \log Z_{N^{1-\varepsilon}}(x)
  >(2\sqrt2+\delta)\sqrt{\log N}
 \right)=0.
\]
\end{proof}
\section{Proof of Proposition~\ref{prop: prop moment bound}: Moments for spatial increments}\label{sec: difference bound}
 Write 
\begin{align}\label{eq:ratioWriting}
\frac{Z_t(x)}{Z_s(x)}-\frac{Z_t(y)}{Z_s(y)}
=\frac{\Delta}{Z_s(x)Z_s(y)},
\qquad
\Delta:=Z_t(x)Z_s(y)-Z_t(y)Z_s(x).
\end{align}
By the Cauchy--Schwarz inequality, it will be enough to bound separately the positive $2p$-th moment of $\Delta$ and the negative moments of $Z_s$.
Sections \ref{subsec:diffusive-paths}--\ref{subsec:small-support} are dedicated to the former. The latter is handled with \eqref{eq:negative-moments}. The conclusion of the proof of Proposition \ref{prop: prop moment bound} is delayed to section \ref{subsec:increment-conclusion}. Without loss of generality, we assume that $x\neq y$ in this section. 

{The proof of the upper bound on moments of $\Delta$  is based on a coupling arguments: we will show that under the polymer measure, paths that start at different points $x,y$ will meet with high probability, at which point they become indistinguishable and then on this event $\Delta$ vanishes by symmetry. The actual argument is significantly more involved.}

\subsection{Restriction to diffusive paths}\label{subsec:diffusive-paths}

The next lemma controls the event that the random walks have many non-diffusive dyadic  intervals.
{Recall that $S_\cdot$ denotes a simple random walk on $\mathbb{Z}^2$ under the law $\mathrm P_0$.} {We call blocks $[2^{l},2^{l+1}]$ in which $|S(2^l)|<L2^{l/2}$ \textit{diffusive}.}
\begin{lemma} \label{lem:remainDiffusive}
For any $\delta>0$, there exist $c,C,L_0>0$ such that, for all $L\geq L_0$ and all integers $1\leq u\leq s$,
\begin{equation}\label{eq: F estimate}
\mathrm P_0\left(|\{l\in [\log_2 u,\log_2 s]\cap \mathbb Z:\Big|S(2^l)\Big|>L2^{l/2} \}| > \delta \log_2(s/u)\right)
\leq C \exp{\left(-cL^2\log_2(s/u)\right)}.
\end{equation}
\end{lemma}
\begin{proof}
Assume first that $s>2u$. If
\[
\big|\{l\in [\log_2 u,\log_2 s]\cap \mathbb Z:|S(2^l)|>L2^{l/2}\}\big|>\delta  \log_2(s/u),
\]
then
\[
\sum_{l\in [\log_2 u,\log_2 s]\cap \mathbb Z}\frac{|S(2^l)|^2}{2^l}
>\delta L^2  \log_2(s/u).
\]
{Therefore, by the exponential Markov inequality with $c_0>0$ sufficiently small together with  the large deviations for dyadic sum of simple random walk \cite[Lemma 7.2.3]{Lawler-Limic}, we obtain  for $L$ sufficiently large,}
\[
\begin{aligned}
&\mathrm P_0\left(
\big|\{l\in [\log_2 u,\log_2 s]\cap \mathbb Z:|S(2^l)|>L2^{l/2}\}\big|>\delta  \log_2(s/u)
\right)
\\
&\qquad\le
\mathrm P_0\left(
\sum_{l\in [\log_2 u,\log_2 s]\cap \mathbb Z}\frac{|S(2^l)|^2}{2^l}
>\delta L^2\log_2(s/u)
\right)
\\
&\qquad\le
\exp(-c_0\delta L^2\log_2(s/u))\,{\mathrm E} _0\exp\left(c_0\sum_{l\in [\log_2 u,\log_2 s]\cap \mathbb Z}\frac{|S(2^l)|^2}{2^l}\right)\\
&\qquad\le
C\exp(-cL^2\log_2(s/u)).
\end{aligned}
\]
This gives \eqref{eq: F estimate} {for $s\geq 2u$. The case $s/u\leq 2$ follows,  after enlarging $C$, from  the  central limit theorem~\cite[Theorem~2.1.3]{Lawler-Limic}.}
\end{proof}

{In what follows, we let $S^1_\cdot$ and $S^2_\cdot$ denote two independent simple random walks on $\mathbb{Z}^2$ started at $x,y$, respectively, and we write $e^1_t,e^2_s$ for the polymer $e_\cdot$ on these respective walks.  We have by definition} $Z_t(x)Z_s(y)=E_{x,y}^{\otimes 2}[e_t^1 e_s^2]$.
 For a path $S$, let
\[
B(S)={B_{u,s,L}(S)}:=
\left|\left\{l\in[\log_2u,\log_2s]\cap\mathbb Z:
\left|S(2^l)-\frac{x+y}{2}\right|>2L2^{l/2}\right\}\right|.
\]
Then, we set 
\[
F_{1,2}:=\{B(S^1)+B(S^2)\le Q_*\},\quad Q_*:=\left \lfloor \frac14\log_2(s/u)\right\rfloor.
\]

From now and until the end of Section \ref{sec: difference bound},
we take $u,s,t,N\in\mathbb N$ and $x,y\in \mathbb Z^2$ satisfying
\begin{equation}\label{eq: condition for ustxy}
\begin{split}
&u := 2^{\lceil\log_2 (|x-y|^2)\rceil } \leq s\leq t\leq N.
\end{split}
\end{equation}

\begin{lemma}\label{lem:bad-event-moment}
There exist $c>0$ {and $L_0>0$} such that, for every integer $p\geq2$, there exists $C=C(p)>0$ such that for every 
$L{\geq L_0}$, and all sufficiently large $N$,
for $u,s,t,x,y$ satisfying \eqref{eq: condition for ustxy},
\begin{align}
\mathbb{E}  \left[\left(E_{x,y}^{\otimes 2}[e_s^1 e_t^2 \1\{F_{1,2}^\complement\}]\right)^p\right]
\leq C \mathbb{E} \big[Z_t(0)^{4p-1}\big]^{1/2} \left(\frac{u}{s}\right)^{cL^2}.
\label{eq-121026c}
\end{align}
The same bound holds for $\mathbb{E}  \left[\left(E_{x,y}^{\otimes 2}[e_t^1 e_s^2 \1\{F_{1,2}^\complement\}]\right)^p\right]$. 
\end{lemma}
\begin{proof}
    Set $A_a:=\{2B(S^a)>\frac14\log_2(s/u)\}$, $a=1,2$.
    Since $F_{1,2}^\complement\subset A_1\cup A_2$, the union bound and
    $(a+b)^p\leq2^{p-1}(a^p+b^p)$ give
    \begin{align*}
        \mathbb{E}  \left[(E_{x,y}^{\otimes 2}[e_s^1 e_t^2 \1\{F_{1,2}^\complement\}])^p\right]
        &\leq  2^{p-1}\sum_{a=1}^2
        \mathbb{E} \left[
        \left(E_{x,y}^{\otimes2}[e_s^1 e_t^2 \1\{A_a\}]\right)^p
        \right]\\
        &\leq {2^p}
        \mathbb{E} \Big[ Z_s(x)^{p}Z_t(y)^{p-(1/2)}
        E_{y}[e_t \1\{2 B(S)>\frac14 \log_2(s/u)\}]^{1/2}\Big].
        \end{align*}

    By Cauchy--Schwarz, H\"older's inequality, and translation invariance,
    \[
    \begin{aligned}
    &\mathbb{E} \Big[ Z_s(x)^{p}Z_t(y)^{p-(1/2)}
    E_y[e_t\mathbf1\{2B(S)>\tfrac14\log_2(s/u)\}]^{1/2}\Big]\\
    &\quad\le
    \mathbb{E} \Big[Z_s(x)^{2p}Z_t(y)^{2p-1}\Big]^{1/2}
    \mathbb{E} \left[E_y\big[e_t\mathbf1\{2B(S)>\tfrac14\log_2(s/u)\}\big]\right]^{1/2}\\
    &\quad\le
    \mathbb{E} \big[Z_t(0)^{4p-1}\big]^{1/2}
    \mathbb{E} \left[E_y\big[e_t\mathbf1\{2B(S)>\tfrac14\log_2(s/u)\}\big]\right]^{1/2},
    \end{aligned}
    \]
where we have used monotonicity of $t\mapsto \mathbb{E} [Z_t^p]$ in the last line.    Moreover,
    \[
    \mathbb{E} \left[E_y\big[e_t\mathbf1\{2B(S)>\tfrac14\log_2(s/u)\}\big]\right]
    =
    P_y\left(2B(S)>\frac14\log_2(s/u)\right).
    \]
    Therefore $\mathbb{E} [(E_{x,y}^{\otimes2}[e_s^1 e_t^2\1\{F_{1,2}^\complement\}])^p]$ is bounded by
    \[
    2^p
    \mathbb{E} \big[Z_t(0)^{4p-1}\big]^{1/2}
    P_y\left(2B(S)>\frac14\log_2(s/u)\right)^{1/2}.
    \]
    The probability
    term is bounded using \eqref{eq: F estimate} with translation invariance of simple random walks and $\left|\frac{x-y}{2}\right|\leq L 2^{l/2}$ for any $l\geq \log_2 u$. {This gives \eqref{eq-121026c}. The claim with the roles of $s,t$ interchanged follows by symmetry.}
\end{proof}

\subsection{Chaos expansion over dyadic blocks}
{We next} estimate 
\[
E_{x,y}^{\otimes 2}\left[e_t^1 e_s^2 \1\{F_{1,2}\}\right]-E_{y,x}^{\otimes 2}\left[e_t^1 e_s^2 \1\{F_{1,2}\}\right].
\]

Recall that
\begin{equation}\label{eq:define-u}
 \frac12 |x-y|^2 \leq   u \leq 2 |x-y|^2.
\end{equation}

Define
\[
  e_{u,t}(S) := \exp\!\left\{\sum_{n=u+1}^{t}
  \left(\beta_N \omega(n, S_n)-\frac{\beta_N^2}{2}\right)\right\}.
\]

We call the interval $(2^i,2^{i+1}]$ the dyadic block with index $i$.
For $s\ge u\ge 1$, define
\[
 \mathcal{J}_{u,s}:=
  \left\{i\in\mathbb Z:\,2^i\in [u,s]\right\}
\]
Let $M:=|\mathcal{J}_{u,s}|=\lfloor\log_2(s/u)\rfloor+1$.  We recall that 
$Q_*:=\left\lfloor \frac14\log_2(s/u)\right\rfloor .
$ 

For every
dyadic block $i\in\mathcal{J}_{u,s}$, define the dyadic weight:
\[
\begin{split}
    &d_i:=d_i(S):=
\exp\left(
\sum_{n\in (u,s]\cap(2^i,2^{i+1}]\cap\mathbb Z}
\left(\beta_N\omega(n,S_n)-\frac{\beta_N^2}{2}\right)
\right),
\end{split}
\]
and $\eta_i:=d_i-1$. For $\alpha\in\{1,2\}$, write $d_i^\alpha = d_i(S^\alpha)$ and $\eta_i^\alpha:=d_i^\alpha-1$, and set
\[
\widetilde\eta_i:=d_i^1d_i^2-1.
\]
Then
\begin{equation}
  \widetilde\eta_i
=(1+\eta_i^1)(1+\eta_i^2)-1
=\eta_i^1+\eta_i^2+\eta_i^1\eta_i^2.
\end{equation}

When considering two random walks, we use {the} shorthand notation:
\[
e_u^{\otimes2}:=e_u(S^1)e_u(S^2),
\qquad
e_{s,t}^\alpha:=e_{s,t}(S^\alpha).
\]

{Finally, write
\[
E_{a,b}^{\otimes 2}\left[e_t^1 e_s^2 \1\{F_{1,2}\}\right] =
{\mathrm E} _{a,b}^{\otimes 2}\left[e_u^{\otimes2} e_{s,t}^1
\prod_{i\in \mathcal J_{u,s}}d_i^1 d_i^2 \1\{F_{1,2}\}\right].
\]
The formula $\prod_{i \in \mathcal I} (1+\tilde \eta_i) = \sum_{I\subset \mathcal I} \prod_{i\in I}\tilde \eta_i$ yields
\[
E_{a,b}^{\otimes 2}\left[e_t^1 e_s^2 \1\{F_{1,2}\}\right] = \sum_{I\subset\mathcal{J}_{u,s}}
{\mathrm E} _{a,b}^{\otimes 2}\left[e_u^{\otimes2} e_{s,t}^1
\prod_{i\in I}\widetilde\eta_i\1\{F_{1,2}\}\right].
\]
We call this the chaos decomposition over dyadic blocks.
}

Define the contributions from large and small supports by
\begin{align*}
  R^{+}_{a,b}&:=
\sum_{\substack{I\subset\mathcal{J}_{u,s}\\ |I|>\varepsilon M}}
{\mathrm E} _{a,b}^{\otimes 2}\left[e_u^{\otimes2} e_{s,t}^1
\prod_{i\in I}\widetilde\eta_i\1\{F_{1,2}\}\right],\qquad
  R^{-}_{a,b}:=
\sum_{\substack{I\subset\mathcal{J}_{u,s}\\ |I|\leq \varepsilon M}}
{\mathrm E} _{a,b}^{\otimes 2}\left[e_u^{\otimes2} e_{s,t}^1
\prod_{i\in I}\widetilde\eta_i\1\{F_{1,2}\}\right],
\end{align*}
so that \(
E_{a,b}^{\otimes 2}\left[e_t^1 e_s^2 \1\{F_{1,2}\}\right] = R_{a,b}^+ + R_{a,b}^-.
\)

\subsection{Chaoses with large support are negligible}
\begin{lemma}\label{lem:step3}
Fix \(L>1\) and \(\varepsilon>0\). For any integer  \(p\geq1\) and
\(A>0\), there exists \(C<\infty\) such that, for all sufficiently
large \(N\), uniformly for \(x\neq y\in\mathbb Z^2\) and \(u,s,t\)
satisfying \eqref{eq: condition for ustxy},
\[
\mathbb{E} \left[\left|R^+_{x,y}\right|^p\right]
\leq
C \,(Q_*+1)
\mathbb{E} \left[Z_t(y)^{8p}\right]^{1/4}
\left(\frac{u}{s}\right)^A .
\]
\end{lemma}

\begin{proof}
Assume \(x\neq y\). 
Removing the support restriction and the path indicators, by $\prod_{i\in I}\tilde \eta_i=\sum_{I'\subset I} \prod_{i\in I'}(-1)^{|I\setminus I'|}d_i^{\otimes 2}$, we have 
\[
\begin{aligned}
\mathbb{E} \left[|R^+_{x,y}|^q\right]
&\leq 2^{2M}\mathbb{E} \left[(Z_t(x)Z_s(y))^q\right] \\
&\leq 2^{2M}
\mathbb{E} \left[Z_t(x)^{2q}\right]^{1/2}
\mathbb{E} \left[Z_s(y)^{2q}\right]^{1/2}
\leq 2^{2M} \mathbb{E} \left[Z_t(y)^{2q}\right].
\end{aligned}
\]

For every fixed \(\delta>0\), if
\(|x-y|/\sqrt{s}\geq\delta\), the claim follows from this moment bound
and H\"older's inequality, after increasing \(C\). We may therefore
assume
\(
\frac{|x-y|}{\sqrt{s}}<\delta,
\) 
with \(\delta>0\) sufficiently small.  
By \eqref{eq:define-u}, this implies that \(s/u\) and $|\mathcal J_{u,s}|$ are larger than a fixed large 
constant depending on $\varepsilon$.

For \(K\subset\mathcal J_{u,s}\), set
\[
Q_+(K):=
\mathbb{E} \left[
\left(
\mathrm {E}\left[
e_u e_{s,t}\prod_{i\in K}\eta_i
\right]
\right)^2
\right].
\]
By Fubini--Tonelli, for every path event \(H\) and \(r\in\{s,t\}\),
\[
\mathbb{E} \left[
\left(
\mathrm {E}\left[
e_u e_{s,r}\prod_{i\in K}\eta_i\mathbf 1_H
\right]
\right)^2
\right]
\leq Q_+(K).
\]

Fix
\(
(a,b)\in\{(x,y),(y,x)\}.
\) 
Using the decompositions
\[
\prod_{i\in I}\widetilde\eta_i
=
\sum_{\substack{
I_1,I_2,I_3\ \mathrm{disjoint}\\
I_1\cup I_2\cup I_3=I}}
\left(\prod_{i\in I_1\cup I_3}\eta_i^1\right)
\left(\prod_{i\in I_2\cup I_3}\eta_i^2\right),
\]
and
\[
\mathbf 1_{F_{1,2}}
=
\sum_{q=0}^{Q_*}
\mathbf 1_{H_q(S^1)}
\mathbf 1_{G_q(S^2)},
\]
where
\[
H_q:=\{B(S)\leq Q_*-q\},
\qquad
G_q:=\{B(S)=q\}.
\]
{The} independence of the walks followed by the Cauchy--Schwarz inequality
give
\begin{equation}\label{eq:Rplus-support-sum-new}
\mathbb {E}\left[|R^+_{x,y}|\right]
\leq
(Q_*+1)
\sum_{\substack{
I_1,I_2,I_3\subset\mathcal J_{u,s}\ \mathrm{disjoint}\\
|I_1|+|I_2|+|I_3|>\varepsilon M}}
Q_+(I_1\cup I_3)^{1/2}
Q_+(I_2\cup I_3)^{1/2}.
\end{equation}

Let $K\subset \mathcal J_{u,s}$ with $|K|> \varepsilon M/2$. We next estimate \(Q_+(K)\). Let \(S,\widetilde S\) be independent walks
started from the same point, and write
\[
I_{a,b}
:=
\beta_N^2\sum_{n=a+1}^{b}
\mathbf 1_{\{S_n=\widetilde S_n\}},
\qquad
I_t:=I_{0,t},
\]
and
\[
\mathcal A_K
:=
\bigcap_{i\in K}
\left\{
\exists n\in(2^i,2^{i+1}]
:
S_n=\widetilde S_n
\right\}.
\]
By Fubini's theorem,
\[
Q_+(K)
\leq
{\mathrm E} \left[e^{I_t}\mathbf 1_{\mathcal A_K}\right].
\]

Fix \(B=B(A,\varepsilon)>0\) and choose \(D\geq1\) sufficiently large, depending only on
\(B\). Both of these constants will be chosen later. Discard from \(K\) the indices satisfying 
\(2^i<2e^D\),  \(2^i>Ne^{-D}\), or \(2^{i+1}>s\). The number of such indices is bounded by a constant depending only on $D$. Let \(K_*\) be the larger of the even and odd subsets
of the remaining indices. Then, since we assumed that $\varepsilon M$ is sufficiently large, we have 
\[
|K_*|\geq \frac{|K|}{4}\geq \frac{\varepsilon M}{8}.
\]

For \(i\in K_*\), let
\[
\tau_i:=\inf\{n>2^i:S_n=\widetilde S_n\}.
\]
Consecutive indices in \(K_*\) differ by at least two. Suppose \(i<i'\)
are consecutive  and \(\tau_i=n\leq2^{i+1}\). For the simple random walks restarted at time \(\tau_i\), let
\[
\rho:=\inf\{r>2^{i'}-n:S_r=\widetilde S_r\}.
\]
On the event \(\{\rho\leq2^{i'+1}-n\}\),
\[
\beta_N^2
\sum_{r=1}^{\rho}
\mathbf 1_{\{S_r=\widetilde S_r\}}
=
\beta_N^2
\sum_{r=1}^{2^{i'}-n}
\mathbf 1_{\{S_r=\widetilde S_r\}}
+\beta_N^2.
\]
Moreover,
\[
e^D\leq2^{i'-1}\leq 2^{i'}-n \leq2^{i'}\leq Ne^{-D},
\qquad
2^{i'+1}-n-(2^{i'}-n)=2^{i'}\leq2(2^{i'}-n).
\]
Let $v:=2^{i'}-n$. Hence, Lemma~\ref{lem:weighted-first-hit} with
\(B=2\) and \(A=D\), yields
\[
\begin{aligned}
&{\mathrm E} _0\left[
\exp\left(
\beta_N^2
\sum_{r=1}^{\rho}
\mathbf 1_{\{S_r=\widetilde S_r\}}
\right)
\mathbf 1_{\{\rho\leq2^{i'+1}-n\}}
\right]\\
&\qquad\leq
e^{\beta_N^2}
{\mathrm E} _0\left[
\exp\left(
\beta_N^2
\sum_{r=1}^{v}
\mathbf 1_{\{S_r=\widetilde S_r\}}
\right)
\mathbf 1_{\{
\exists r\in(v,3 v]:
S_r=\widetilde S_r
\}}
\right] \leq
\frac{C_2}{D},
\end{aligned}
\]
where \(C_2\)  is independent of \(D\).

Let \(i_0:=\min K_*\) and \(v:=2^{i_0}\).
Since
\(
e^D\leq v\leq Ne^{-D},
\) 
Lemma~\ref{lem:weighted-first-hit} gives
\[
{\mathrm E} _0\left[
e^{I_{\tau_{i_0}}}
\mathbf 1_{\{\tau_{i_0}\leq2v\}}
\right]
=
e^{\beta_N^2}
{\mathrm E} _0\left[
e^{I_v}
\mathbf 1_{\{
\exists n\in(v,2v]:
S_n=\widetilde S_n
\}}
\right]
\leq
\frac{C_2}{D}.
\]

Order \(K_*=\{i_0<\cdots<i_m\}\). Applying the strong Markov property successively at \(\tau_{i_0},\ldots,\tau_{i_m}\), each transition to the next selected block contributes at most \(C_2/D\) by the estimate above, while the contribution after the last collision is bounded by \(\mathbb{E} [Z_t^2]\); hence
\[
Q_+(K)\le C\,\mathbb{E} [Z_t^2]\left(\frac{C_2}{D}\right)^{|K_*|},
\] 
with some constant  $C$ that may depend on $D$.
 
 When \(|K|< \varepsilon M/2\), we use the bound
\[
Q_+(K)\leq {\mathrm E} [e^{I_t}]
=\mathbb{E} [Z_t^2].
\]

For a triple {as} in \eqref{eq:Rplus-support-sum-new}, set
\[
K_1:=I_1\cup I_3,
\qquad
K_2:=I_2\cup I_3.
\]
Then
\[
|K_1|+|K_2|
=
|I_1|+|I_2|+2|I_3|
>
\varepsilon M.
\]
Together with the bound \(4^M\) on
the number of disjoint triples, if we take $D$ sufficiently large depending on $B$, we obtain 
\[
\mathbb{E} {\left[|R^+_{x,y}|\right]}
\leq
4^M C(Q_*+1)\mathbb{E} [Z_t^2]
\exp\left(
 - B\varepsilon M
\right).
\]

Choosing \(B\) sufficiently large depending on $A$ and $\varepsilon$, since
\(
M\geq\log_2(s/u),
\) 
we obtain
\begin{equation}\label{eq:Rplus-L1-improved}
\mathbb{E} {\left[|R^+_{x,y}|\right]}
\leq
C\,2^{-4M}\,(Q_*+1)\mathbb{E} [Z_t^2]
\left(\frac{u}{s}\right)^{2A} .
\end{equation}
This gives 
\[
\begin{aligned}
\mathbb{E} \left[|R^+_{x,y}|^p\right]
&\leq
\mathbb{E} \left[|R^+_{x,y}|^{4p-2}\right]^{1/4}
\mathbb{E} \left[|R^+_{x,y}|\right]^{1/2}\\
&\leq
C\,(Q_*+1)
\mathbb{E} \left[Z_t(y)^{8p-4}\right]^{1/4}
\mathbb{E} [Z_t(y)^2]^{1/2}
\left(\frac{u}{s}\right)^{A}.
\end{aligned}
\]
By H\"older's inequality, we have 
\[
\mathbb{E} [Z_t^2]^{1/2}
\leq
\mathbb{E} [Z_t^{8p}]^{1/(8p)},
\quad 
\mathbb{E} [Z_t^{8p-4}]^{1/4}
\leq
\mathbb{E} [Z_t^{8p}]^{
\frac{4p-2}{16p}
}.
\]
Since
\(
\frac{1}{8p}
+
\frac{4p-2}{16p}
=
\frac14,
\) 
it follows that
\[
\mathbb{E} [Z_t^{8p-4}]^{1/4}
\mathbb{E} [Z_t^2]^{1/2}
\leq
\mathbb{E} [Z_t^{8p}]^{1/4}.
\] 
Therefore
\[
\mathbb{E} \left[|R^+_{x,y}|^p\right]
\leq
C(Q_*+1)
\mathbb{E} \left[Z_t(y)^{8p}\right]^{1/4}
\left(\frac{u}{s}\right)^{A}.
\]
This proves the claim.
\end{proof}
\subsection{Swapping chaoses with small support}\label{subsec:small-support}
In this subsection, we estimate the difference \(R^-_{x,y}-R^-_{y,x}\). We restrict to terms where there are at least \(|\mathcal J_{u,s}|/8\) blocks in which the random walks traverse diffusively {(defined above Lemma \ref{lem:remainDiffusive})} but no chaos is present. We call such blocks \emph{diffusive, no-chaos blocks}. The use of these blocks is inspired by the separation of the environments near the starting and ending points in the local limit theorem arguments of Sinai~\cite{Sinai95} and Junk~\cite{Junk25}.  The key observation is that, in each such block, there is a uniformly positive probability of exchanging the roles of the starting points \(x\) and \(y\). Once such an exchange occurs, the corresponding contribution to \(R^-_{x,y}-R^-_{y,x}\) cancels. By the swapping lemma, Lemma~\ref{lem: swapping lemma}, the contribution from paths for which no exchange occurs in any of the diffusive, no-chaos blocks is therefore bounded by
\[
(1-c)^{|\mathcal J_{u,s}|/8}\leq (u/s)^{\alpha}
\]
for some \(c,\alpha>0\). 
  
{Turning to the actual argument, define}
\begin{align}\label{eq: def of bar Z}
  \bar Z_r^{||}(x) := \sum_{J\subset \mathcal{J}_{u,s}, |J|\leq \varepsilon M} {\mathrm E} _x\left[e_u  e_{s,r} \prod_{j\in J} d_j\right].
\end{align}

\begin{lemma}\label{lem: swapping small-chaos}
For every $L>1$, there exist $\alpha,C>0$ and $\varepsilon_0 \in (0,1/16)$ such that, for all  $\varepsilon\in (0,\varepsilon_0)$, the following holds. Let $x,y\in\mathbb Z^2$
satisfy $x-y\in4\mathbb Z^2$, and let $u,s,t$ satisfy \eqref{eq: condition for ustxy}.
Then, {almost surely,}
\begin{align*}
&  \left|R^{-}_{x,y}-R^{-}_{y,x}\right| \;\le\; C\bigl(\bar Z_t^{||}(x)\bar Z^{||}_s(y)+\bar Z_t^{||}(y)\bar Z^{||}_s(x)\bigr)\, \left(\frac{u}{{s}}\right)^{\alpha}.
\end{align*}
\end{lemma}
\begin{proof} 
Write $d_i^{\otimes2}:=d_i^1d_i^2$. Recall that $M=|\mathcal{J}_{u,s}|=\lfloor\log_2(s/u)\rfloor+1$. We take $M_0=M_0(L)$ sufficiently large and $\varepsilon_0=\varepsilon_0(L)$ sufficiently small to be chosen later. We assume that $\varepsilon<\varepsilon_0$. After increasing $C$, we may assume $M>M_0$. {Set $z_\star:=(x+y)/2$. Since $x-y\in 4\mathbb Z^2$, we have
$z_\star=x+(y-x)/2\in\mathbb Z^2$, and $z_\star$ has the same parity as
$x$.}

  {\textit{Step 1.} } We fix
$J\subset \mathcal{J}_{u,s}$ with $|J|\le\lfloor\varepsilon M\rfloor$.
For $z,w\in\mathbb Z^2$, let
\begin{equation}
\label{eq-161026a}
Z_{x,y}^J(z,w)
:=
{\mathrm E} ^{\otimes2}_{x,y}\left[
e_u^{\otimes2}\prod_{j\in J}d_j^{\otimes2}
 \mathbf1_{\{F_{1,2},\,S_s^1=z,\,S_s^2=w\}}\right].
\end{equation}

We claim that there exists $c_0\in (0,1)$ such that for all $z,w\in\mathbb Z^2$, almost surely,
\begin{equation}\label{eq:fixed-endpoint-swap}
\left|Z_{x,y}^J(z,w)-Z_{y,x}^J(z,w)\right|
\le
(1-c_0)^{M/8}
\left(Z_{x,y}^J(z,w)+Z_{y,x}^J(z,w)\right).
\end{equation}

Set
\[
 \iota(n):=\left\lceil \log_2 n\right\rceil-1.
\]
Note that $2^{\iota(s)}<s\leq2^{\iota(s)+1}$.
For integers $0\leq a<b$ and
$v,w\in\mathbb Z^2$, set
\[
Z_{(a,b]}(v,w):=
{\mathrm E} \left[
\exp\left(
\sum_{n=a+1}^{b}
\left(\beta_N\omega(n,S_n)-\frac{\beta_N^2}{2}\right)
\right)
\mathbf1_{\{S_b=w\}}
\mid S_a=v
\right].
\]
{Recall that}
$p_n(x,y):=\mathrm P_x(S_n=y)=p_n(y-x)$. Put $x_\ell:=S^1_{2^\ell}$, $y_\ell:=S^2_{2^\ell}$,
$x_{\iota(s)+1}:=z$, and $y_{\iota(s)+1}:=w$. Let $\mathcal A_{z,w}$ be the finite
set of sequences
\[
U=(x_0,y_0,\ldots,x_{\iota(s)},y_{\iota(s)})
\]
satisfying
$$
x_\ell,y_\ell\in
\{v\in\mathbb Z^2:p_{2^\ell}(x,v)+p_{2^\ell}(y,v)>0\},
\qquad 0\leq\ell\leq \iota(s),
$$ 
and (recalling that $Q_*= \left\lfloor \frac14\log_2(s/u) \right\rfloor$)
\[
\sum_{q\in[\log_2u,\log_2s]\cap\mathbb Z}
\left(
\mathbf1_{\{|x_q-{z_\star}|>2L2^{q/2}\}}
+\mathbf1_{\{|y_q-{z_\star}|>2L2^{q/2}\}}
\right)
\leq Q_*.
\]
Note that $\mathcal A_{z,w}$ encodes the event $F_{1,2}$. For
$U\in\mathcal A_{z,w}$, define
\[
I_U:=
\left\{
\ell\in\{Q_*,\ldots,\iota(s)\}:
\begin{array}{l}
    \ell \notin J,\, 2^\ell \ge u,\, 2^{\ell+1}\leq s,\\
      \sum_{v\in\{x_\ell,y_\ell,x_{\ell+1},y_{\ell+1}\}}
 \left|v-{z_\star}\right|\le 12 L 2^{\ell/2}
 \end{array}
 \right\}.
\]
Since we assumed that $M$ is sufficiently large and $|J|\leq \varepsilon M$, we have
\[
|I_U|\ge M - |J| - 3Q_* -1\geq \frac M8.
\]

 For $q\in\{0,\ldots,\iota(s)\}$, define  the swapping operator $\tau_q$ {up to index $q$ by the following: for for $U=(x_0,y_0,\ldots,x_{\iota(s)},y_{\iota(s)})$ with $(x_\ell)_{\ell=0}^{\iota(s)},$ and $ (y_\ell)_{\ell=0}^{\iota(s)}$,}
$$\tau_q(U) = (y_0,x_0,\dots,y_q,x_q,x_{q+1},y_{q+1},\dots, x_{\iota(s)},y_{\iota(s)})
 $$
Then
\[
\tau_q\mathcal A_{z,w}=\mathcal A_{z,w},
\qquad I_{\tau_qU}=I_U.
\]

For $0\leq\ell\leq \iota(s)$, {let  
$\rho_\ell(v):=p_{2^\ell}(z_\star,v)$ and} set, {for $x',y'\in \mathbb Z^2$,}
\[
f_\ell(x',y'):=
\begin{cases}
Z_{(2^\ell,2^{\ell+1}\wedge s]}({x',y'}),
&\ell\in J\ \text{or}\ 2^{\ell}< u,\\
p_{(2^{\ell+1}\wedge s)-2^\ell}({x',y'}),
&\ell\notin J\ \text{and}\ 2^\ell\ge u.
\end{cases}
\]
 The  local central limit theorem for simple random walks {\cite[Theorem~2.1.3]{Lawler-Limic}} yields that there exist constants $M_0=M_0(L)>0$ and  
$0<\kappa_L\le K_L<\infty$ such that, for any $M>M_0$ and for $\ell\in I_U$,
\[
\kappa_L \rho_\ell(x_{\ell+1})\rho_\ell(y_{\ell+1})\le f_\ell(x_\ell,x_{\ell+1})f_\ell(y_\ell,y_{\ell+1})
\le
K_L \rho_\ell(x_{\ell+1})\rho_\ell(y_{\ell+1}).
\]
Set
\[
r_\ell:=
\begin{cases}
\sqrt{({\kappa_L}/2)}\,\rho_\ell,
&\ell\notin J,\ 2^\ell\ge u,\ 2^{\ell+1}\leq s,\\
0,&\text{otherwise}.
\end{cases}
\qquad c_0:=\kappa_L/(2 K_L).
\]
{Then,} for any $\ell\in I_U$,
\begin{align}
  0&\le
f_\ell(x_\ell,x_{\ell+1})f_\ell(y_\ell,y_{\ell+1})
-r_\ell(x_{\ell+1})r_\ell(y_{\ell+1})\notag\\
&\le
(1-c_0)f_\ell(x_\ell,x_{\ell+1})f_\ell(y_\ell,y_{\ell+1}).\label{eq: comparison between f and rho}
\end{align}

We write $Z^J_{x,y}$ and $Z^J_{y,x}$  {of \eqref{eq-161026a}} as 
\[
Z_{x,y}^J(z,w)=
\sum_{U\in\mathcal A_{z,w}}\mathsf A(U)
\prod_{\ell=0}^{\iota(s)}\mathsf F_\ell(U),
\qquad
Z_{y,x}^J(z,w)=
\sum_{U\in\mathcal A_{z,w}}\mathsf B(U)
\prod_{\ell=0}^{\iota(s)}\mathsf F_\ell(U),
\]
where
\[
\mathsf A(U):=Z_{(0,1]}(x,x_0)Z_{(0,1]}(y,y_0),\qquad
\mathsf B(U):=Z_{(0,1]}(y,x_0)Z_{(0,1]}(x,y_0),
\]
\[
\mathsf F_\ell(U):=
f_\ell(x_\ell,x_{\ell+1})f_\ell(y_\ell,y_{\ell+1}),
\qquad
\mathsf K_\ell(U):=
r_\ell(x_{\ell+1})r_\ell(y_{\ell+1}).
\]
Note that \[
(\mathsf A-\mathsf B)\circ\tau_q=-(\mathsf A-\mathsf B),\qquad
\mathsf F_\ell\circ\tau_q=\mathsf F_\ell\quad(\ell\ne q),
\qquad
\mathsf K_\ell\circ\tau_q=\mathsf K_\ell\quad(\ell\le q).
\] 
We define 
\[
\Gamma_\ell(U)
:=
f_\ell(x_\ell,x_{\ell+1})f_\ell(y_\ell,y_{\ell+1})
-r_\ell(x_{\ell+1})r_\ell(y_{\ell+1}).
\]
Lemma~\ref{lem: swapping lemma} with $\mathcal A = \mathcal A_{z,w}$ and $m=\iota(s)$  tells us that 
\begin{align*}
   & \sum_{U\in\mathcal A_{z,w}} \mathsf A(U)\prod_{\ell=0}^{\iota(s)} \mathsf F_\ell(U)
-\sum_{U\in\mathcal A_{z,w}} \mathsf B(U)\prod_{\ell=0}^{\iota(s)} \mathsf F_\ell(U)\\
&\qquad=
\sum_{U\in\mathcal A_{z,w}} \mathsf A(U)
\prod_{\ell\in I_U}\Gamma_\ell(U)
\prod_{\ell\in\{0,\dots,\iota(s)\}\setminus I_U} \mathsf F_\ell(U) \\
&\qquad \qquad -
\sum_{U\in\mathcal A_{z,w}} \mathsf B(U)
\prod_{\ell\in I_U}\Gamma_\ell(U)
\prod_{\ell\in\{0,\dots,\iota(s)\}\setminus I_U} \mathsf F_\ell(U).
\end{align*}

All factors on the right-hand side  are
nonnegative. For each $U\in\mathcal A_{z,w}$, the factor $\prod_{\ell\in I_U}\Gamma_\ell(U)$ is bounded by
$(1-c_0)^{|I_U|}\prod_{\ell\in I_U} \mathsf F_\ell(U)$ with some constant $c_0>0$ due to \eqref{eq: comparison between f and rho}. Therefore, since
$|I_U|\geq M/8$, we obtain \eqref{eq:fixed-endpoint-swap}.

{\textit{Step 2.} }
For $I\subset\mathcal{J}_{u,s}$, define
\[
\Delta_I:=
{\mathrm E} ^{\otimes2}_{x,y}\left[e_u^{\otimes2}e_{s,t}^1
\Bigl(\prod_{i\in I}\widetilde\eta_i\Bigr)
\mathbf1_{F_{1,2}}\right]
-
{\mathrm E} ^{\otimes2}_{y,x}\left[e_u^{\otimes2}e_{s,t}^1
\Bigl(\prod_{i\in I}\widetilde\eta_i\Bigr)
\mathbf1_{F_{1,2}}\right].
\]
Note that we did not apply the swapping lemma, Lemma~\ref{lem: swapping lemma}, directly to $\Delta_I$. Indeed, since $\tilde{\eta}_i$ may be negative, {if we did that  the resulting bound would involve}  $|\tilde{\eta
}_i|$, which is too costly for our purposes.

The expansion above gives
\[
\Delta_I=\sum_{J\subset I}(-1)^{|I|-|J|}\tilde{\Delta}^J,
\]
where
\[
\tilde{\Delta}^J:=
{\mathrm E} ^{\otimes2}_{x,y}\left[e_u^{\otimes2}\Bigl(\prod_{j\in J}d_j^{\otimes2}\Bigr)
e_{s,t}^1\mathbf1_{F_{1,2}}\right]
-{\mathrm E} ^{\otimes2}_{y,x}\left[e_u^{\otimes2}\Bigl(\prod_{j\in J}d_j^{\otimes2}\Bigr)
e_{s,t}^1\mathbf1_{F_{1,2}}\right].
\]
We  apply
\eqref{eq:fixed-endpoint-swap} to every $\tilde{\Delta}^J$.  After summing over the
endpoints at time $s$, the two
positive terms in \eqref{eq:fixed-endpoint-swap} are bounded by the
corresponding products defining $\bar Z^{||}$.  Moreover,
$\widetilde\eta_i=d_i^{\otimes2}-1$, so the expansion of a fixed $I$
produces $2^{|I|}$ terms $\tilde{\Delta}^J$. Consequently,
\begin{align*}
\left|R^{-}_{x,y}-R^{-}_{y,x}\right|&\leq \sum_{\substack{I\subset\mathcal{J}_{u,s}\\ |I|\le\varepsilon M}}|\Delta_I|
\le \sum_{\substack{I\subset\mathcal{J}_{u,s}\\ |I|\le\varepsilon M}}\sum_{J\subset I}|\tilde{\Delta}^J| \\
&\le C\bigl(\bar Z_t^{||}(x)\bar Z_s^{||}(y)
+\bar Z_t^{||}(y)\bar Z_s^{||}(x)\bigr)
(1-c_0)^{M/8}
\sum_{r\leq\varepsilon M}\binom{M}{r}2^r.
\end{align*} 
By Stirling's formula, {for all $\varepsilon$ sufficiently,}
\[
\sum_{r\le \varepsilon M}\binom{M}{r}2^r 
\le \exp({2}\varepsilon \log(1/\varepsilon)\,M).
\]
Choose $\varepsilon>0$ so that
${2}\varepsilon\log(1/\varepsilon)<-\frac1{16}\log(1-c_0)$. The product of the
combinatorial factor and the contraction factor is then at most
$Ce^{\frac1{16}\log(1-c_0)M}$. Since
\(
M\ge \log_2(s/u),
\)   this 
 gives
\[
e^{\frac1{16}\log(1-c_0)M}\le \left(\frac{u}{s}\right)^\alpha,
\]
where we set $\alpha:=-(16\log2)^{-1} \log (1-c_0)$. This proves the lemma.
\end{proof}
{We also need a control on the moments of $\bar Z_t^{||}(x)$, see \eqref{eq: def of bar Z}.}
\begin{lemma} \label{lem:pmomentZparallel}
There exists $C>1$ such that, for every fixed  integer $p\geq2$ and for all sufficiently large $N$, $x,y\in {\rm B}(\sqrt N)$,
and $u,s,t$ satisfying \eqref{eq: condition for ustxy},
\begin{equation}\label{eq:Zbar-positive-moment}
\mathbb{E} \left[\bar Z_t^{||}(x)^p\right]
\leq C\left(s/u\right)^{Cp\varepsilon\log(1/\varepsilon)}
\mathbb{E} [Z_t(0)^p].
\end{equation}
\end{lemma}
\begin{proof}
We first note that 
\[
\#\{J\subset \mathcal{J}_{u,s}: |J|\le \varepsilon M\}
\le \exp(C\varepsilon\log(1/\varepsilon)M)
\le
C\left({{s}}/{u}\right)^{C\varepsilon\log(1/\varepsilon)}.
\]
For each fixed $J$, the chaos expansion (see \cite[Appendix~A.1]{CSZ-Dickman} for instance) gives
$$
\begin{aligned}
&\mathbb{E} \Big[\Big({\mathrm E} _x[e_u\prod_{j\in J}d_j]\Big)^p\Big],\quad
\mathbb{E} \Big[\Big({\mathrm E} _x[e_u e_{s,t}\prod_{j\in J}d_j]\Big)^p\Big]\leq\mathbb{E} [Z_t(0)^p].
\end{aligned}
$$
H\"older's inequality gives
\[
\mathbb{E} \left[\bar Z_t^{||}(x)^p\right]
\leq
C\left({s}/u\right)^{C p\varepsilon \log(1/\varepsilon)}
\mathbb{E} [Z_t(0)^p].
\]
\end{proof}
\subsection{Conclusion of the proof of Proposition~\ref{prop: prop moment bound}}\label{subsec:increment-conclusion}
By Corollary~\ref{cor:positive-moment-bound}, we have the following bound 
\begin{align}\label{eq: record moment}
    \sup_{T\leq N}\mathbb{E} [Z_T^p]\leq \exp\left(C(\log\log N)^2\right).
\end{align}

{By taking  $s:=t_k$ and $t:=t_{k+1}$ in \eqref{eq:ratioWriting}, it is enough to bound
\[
\mathbb{E}  \left[\left|\frac{\Delta}{Z_s(x)Z_s(y)}\right|^p\right], \quad \Delta:=Z_t(x)Z_s(y)-Z_t(y)Z_s(x).
\]
By the Cauchy--Schwarz inequality, we have
\begin{align}\label{eq:reduce-to-delta}
\mathbb{E} \left[
\left|
\frac{\Delta}{Z_s(x)Z_s(y)}
\right|^p
\right]
&\le
\mathbb{E} [|\Delta|^{2p}]^{1/2}
\mathbb{E} [Z_s^{-4p}]^{1/2}.
\end{align}
By \eqref{eq:negative-moments} with \eqref{eq:2ndMomentBoundIntroduction}, we have $\mathbb{E} [Z_s^{-4p}] \leq e^{C_p k}$. 

It remains to bound $\mathbb{E} |\Delta|^{2p}$.  
With the above notation,
\begin{align*}
\Delta
&=
\bigl(R^{+}_{x,y}-R^{+}_{y,x}\bigr)
+\bigl(R^{-}_{x,y}-R^{-}_{y,x}\bigr)
+\Delta_{F^\complement},
\end{align*}
where
\begin{align*}
\Delta_{F^\complement}
&:=
{\mathrm E} ^{\otimes2}_{x,y}\left[e_t^1 e_s^2\mathbf1_{F_{1,2}^\complement}\right]
-{\mathrm E} ^{\otimes2}_{y,x}\left[e_t^1 e_s^2\mathbf1_{F_{1,2}^\complement}\right].
\end{align*}

Let $L_0$ be as in Lemma \ref{lem:bad-event-moment}. Then, by Lemma \ref{lem:bad-event-moment} and Lemma \ref{lem:step3}, there exists $C_p,C_{p,L}>0$ such that for all $L> L_0$ sufficiently large,  
\begin{align*}
\mathbb{E}  \left[\left|\Delta_{F^\complement}\right|^{2p}\right]
&\leq \exp(C_{p}(\log\log N)^2)
\left(\frac{u}{s}\right)^{2A} ,\\
\mathbb{E} \left[\left|R^+_{x,y}-R^+_{y,x}\right|^{2p}\right]
&\leq \exp(C_{p,L}(\log\log N)^2)
\left(\frac{u}{s}\right)^{2A}.
\end{align*}
We fix such a $L>L_0$. 

Let $\alpha>0$ from Lemma \ref{lem: swapping small-chaos}. Since $x\neq y$ and $|x-y|\leq\sqrt{t_k/2}$, we have $|x-y|^2\leq u:=2^{\lceil\log_2(|x-y|^2)\rceil}<2|x-y|^2\leq t_k$, which verifies \eqref{eq: condition for ustxy}. Combining Lemma \ref{lem: swapping small-chaos} with Lemma \ref{lem:pmomentZparallel} yields
the existence of $C_p$ and  $\varepsilon=\varepsilon(p,\alpha) \in (0,1)$,  such that
\[
\mathbb{E}  \left[\left|R^{-}_{x,y}-R^{-}_{y,x}\right|^{2p}\right] \leq C_p \exp(C_p(\log\log N)^2) \left(\frac{u}{{s}}\right)^{{\alpha p}},
\]
where we have chosen $\varepsilon$ sufficiently small so that
$4C\varepsilon\log(1/\varepsilon)\leq\alpha$ and used $u/s\leq1$. Taking an integer $p$ sufficiently large depending on $A,\alpha$ and combining the three estimates,  
we obtain 
\begin{equation}
\mathbb{E} [|\Delta|^{2p}] \leq \exp(C_{p}(\log\log N)^2)
\left(\frac{u}{s}\right)^{2A}.
\end{equation}
}
Substituting the 
 estimate above into \eqref{eq:reduce-to-delta} yields, {for some constant $C$ depending on $p$,}
\begin{align*}
\mathbb{E} \left[
\left|
\frac{Z_t(x)}{Z_s(x)}-\frac{Z_t(y)}{Z_s(y)}
\right|^p
\right]
\le
\exp(Ck+C(\log\log N)^2)\left(\frac{u}{{s}}\right)^{A}.
\end{align*}
{Since}  $u\asymp |x-y|^2$, this gives \eqref{eq:8times}.\hfill \qed

\section{{Proof of the lower bound in Theorem \ref{thm:critical-maximum}}}
\label{sec-LB}
We prove that, for every $\delta>0$,
\begin{equation}\label{eq:lower-bound-goal}
\Pbb\left(
\max_{x\in {\rm B}(\sqrt N)}\log Z_N(x)
<
(2\sqrt2-\delta)\sqrt{\log N}
\right)\longrightarrow 0 .
\end{equation}
We use the following subcritical estimate in \cite{CNZ-max}: for every $\delta>0$, there
exists a sufficiently small $\varepsilon>0$ such that, with
$N':=\lfloor N^{1-\varepsilon}\rfloor$,
\begin{equation}\label{eq:subcritical-lower-input}
\Pbb\left(
\exists x\in {\rm B}(\sqrt N):
\log Z_{N'}(x)>
\left(2\sqrt2-\frac{\delta}{2}\right)\sqrt{\log N}
\right)\longrightarrow 1 .
\end{equation}
Indeed, $\beta_N^2R_{N'}=1-\varepsilon+O((\log N)^{-1})$.
Then, the main result in \cite{CNZ-max} with change of variables \eqref{eq: transformation of Zt} gives, on the smaller box
 ${\rm B}(\sqrt{N'})$,
\begin{align*}
    (\sqrt{\log N})^{-1}{\max_{x\in {\rm B}(\sqrt{N'})}\log Z_{N'}(x)}
&\xrightarrow{\mathbb P}
\sqrt{2(1-\varepsilon)}
\int_0^1\sqrt{\frac{1-\varepsilon}{1-(1-\varepsilon)v}}\,dv\\
&=2\sqrt2(1-\sqrt\varepsilon).
\end{align*}
Choosing $\varepsilon$ so that
$2\sqrt2\sqrt\varepsilon<\delta/2$, and using
${\rm B}(\sqrt{N'})\subset {\rm B}(\sqrt N)$, this gives
\eqref{eq:subcritical-lower-input}.

For $m<n$, set
$$
Z_{[m,n]}(m,y)
:=
{\mathrm E} _y\left[
\exp\left(
\sum_{j=m+1}^{n}
\left(\beta_N\omega(j,S_{j-m})-\frac{\beta_N^2}{2}\right)
\right)
\right],
$$
and let $\F_m:=\sigma(\omega(j,z):1\le j\le m,\ z\in\Zbb^2)$.
We also write $\mu_m^x$ for the polymer endpoint law at time $m$ started
from $x$.

Fix $\delta>0$, choose $\varepsilon$ so that
\eqref{eq:subcritical-lower-input} holds. Let
$$
\mathcal E_1:=
\left\{
\exists x\in {\rm B}(\sqrt N):
\log Z_{N'}(x)>
\left(2\sqrt2-\frac{\delta}{2}\right)\sqrt{\log N}
\right\}.
$$
On the event $\mathcal E_1$, choose an $\F_{N'}$-measurable point
$x_*\in {\rm B}(\sqrt N)$ satisfying
$\log Z_{N'}(x_*)>(2\sqrt2-\delta/2)\sqrt{\log N}$, and define it
arbitrarily on $\mathcal E_1^\complement$. By the Markov property,
\begin{equation}\label{eq:ratio-after-cut}
\frac{Z_N(x_*)}{Z_{N'}(x_*)}
=
\sum_y \mu_{N'}^{x_*}(y)\,Z_{[N',N]}(N',y).
\end{equation}
By Corollary~\ref{cor:positive-moment-bound}, there
exists $C_2>0$ such that $\mathbb{E} [Z_N^2]\leq
\exp(C_2(\log\log N)^2)$. Fix $\alpha>C_2/2$.
Since $\mathbb{E} [Z_{N-N'}^2]\leq\mathbb{E} [Z_N^2]$, Proposition~\ref{prop:iterated-lower-tail} with exponential Makov's inequality, applied with $T=N$, gives
\begin{equation}\label{eq:remaining-lower-tail}
\Pbb\left(\log Z_{N-N'}(0)<-\alpha(\log\log N)^2\right)
\leq\exp\left(-\frac{(\alpha-C_2/2)^2}{2C_2}
(\log\log N)^2\right)\longrightarrow0.
\end{equation}
Since the future environment is independent of $\F_{N'}$ and is shift
invariant, for every $y$,
$$
\begin{aligned}
&\Pbb\left(
\log Z_{[N',N]}(N',y)<-\alpha(\log\log N)^2
\,\middle|\,\F_{N'}\right)
=\Pbb\left(
\log Z_{N-N'}(0)<-\alpha(\log\log N)^2\right).
\end{aligned}
$$
Define
$$
\begin{aligned}
  &  \mathcal R_N:=
\sum_y \mu_{N'}^{x_*}(y)
\1\{\log Z_{[N',N]}(N',y)<-\alpha(\log\log N)^2\},\quad \mathcal E_2:=\left\{\mathcal R_N\ge \frac12\right\}.
\end{aligned}
$$
Then
$$
\E\left[\mathcal R_N\,\middle|\,\F_{N'}\right]
=\Pbb\left(
\log Z_{N-N'}(0)<- \alpha(\log\log N)^2
\right),
$$
and Markov's inequality gives
\begin{equation}\label{eq:E2-bound}
\Pbb\left(\mathcal E_2\,\middle|\,\F_{N'}\right)
\le 2\Pbb\left(
\log Z_{N-N'}(0)<- \alpha(\log\log N)^2
\right).
\end{equation}
Consequently,
$$
\Pbb(\mathcal E_1\cap \mathcal E_2^\complement)
\ge
1-\Pbb(\mathcal E_1^\complement)-2\Pbb\left(
\log Z_{N-N'}(0)<- \alpha(\log\log N)^2
\right)
\longrightarrow 1 .
$$
On $\mathcal E_1\cap\mathcal E_2^\complement$,
\eqref{eq:ratio-after-cut} gives
$$\!
\frac{Z_N(x_*)}{Z_{N'}(x_*)}
\geq e^{-\alpha(\log\log N)^2}
\sum_y\mu_{N'}^{x_*}(y)
\1\{\log Z_{[N',N]}(N',y)\!\geq\!-\alpha(\log\log N)^2\}
\!\!\geq\!\frac12e^{-\alpha(\log\log N)^2}.
$$
By the definition of $\mathcal E_1$, {this implies that}
\begin{align*}
    \log Z_N(x_*)
&\ge
\left(2\sqrt2-\frac{\delta}{2}\right)\sqrt{\log N}-\alpha(\log\log N)^2-\log 2
\geq
\left(2\sqrt2-\delta\right)\sqrt{\log N}.
\end{align*}
Since
$x_*\in {\rm B}(\sqrt N)$, this proves \eqref{eq:lower-bound-goal}.

\begin{appendices}
\section{Maximum  in subcritical regime: Proof of Lemma~\ref{lem:CNZ-upper-tail}}
\label{app-pflem3.4}
\begin{proof}
{Due to spatial translation invariance of the law of the environment, it is enough to consider $z=(0,0)$.}
We use the barrier argument of \cite[Section 3]{CNZ-max}.
Put $n:=\lfloor N^{1-\varepsilon}\rfloor$.
Throughout the proof, $\varepsilon$ is fixed before $N\to\infty$,
and we suppress the superscript $\beta_N$ on partition functions.
Choose $0<\rho<1/3$ so small that $2\sqrt2\rho<\eta/8$, and then
$A:=A(\eta)>16$ so large that
$$
 (1-\rho)^2<\frac{1}{1+2/A},
 \qquad \frac{\sqrt2(1+\rho)}{A}<\frac\eta8.
$$
Set $M:=\lceil A/\varepsilon\rceil$ {(not to be confused with $M$ in Section \ref{sec: difference bound})}, $s_0:=0$, and, for
$1\leq k\leq M$, set
$$
 s_k:=\lceil n^{k/M}\rceil,\qquad
 W_{k,N}(x):=\frac{Z_{s_k}(x)}{Z_{s_{k-1}}(x)},
$$
and
$$
\lambda^2_k= \lambda_{k,\varepsilon,N}^2
 :=\log\frac{1-\widehat\beta_{N,\varepsilon}^2(k-1)/M}
                 {1-\widehat\beta_{N,\varepsilon}^2k/M}.
$$
Thus $\log Z_n(x)=\sum_{k=1}^M\log W_{k,N}(x)$, since $Z_0=1$. {Note also that $\lambda_k$ is bounded above by a constant independent of $N$, and in fact due to our choice of $M$, by a constant that is independent of $\varepsilon$.}

We first bound the barrier height. The limiting variance profile is
$$
 \sigma_\varepsilon(u):=
 \sqrt{\frac{1-\varepsilon}{1-(1-\varepsilon)u}},
 \qquad 0\leq u\leq1,
$$
and
\begin{equation}
\label{eq:subcritical-coeff}
 \sqrt{1-\varepsilon}\,\sqrt2
 \int_0^1\sigma_\varepsilon(u)\,du
 =2\sqrt2(1-\sqrt\varepsilon).
\end{equation}
Since $\sigma_\varepsilon$ is increasing, comparison with the right
Riemann sum gives, for fixed $\varepsilon,M$,
$$
 \frac{\sqrt2}{\sqrt M}\sum_{k=1}^M\lambda_k
 \leq\sqrt2\int_0^1\sigma_\varepsilon(u)\,du
       +\frac{\sqrt2}{M\sqrt\varepsilon}+o_N(1).
$$
Define
$$
 f_k:=\frac{\sqrt2(1+\rho)}{\sqrt M}\lambda_k\sqrt{(1-\varepsilon)\log N},
 \qquad a:=\frac\eta4\sqrt{(1-\varepsilon)\log N}.
$$
By our choices of $\rho,A$, for $0<\varepsilon<1/2$ and
$N$ sufficiently large,
$$
 \sum_{k=1}^Mf_k\leq(2\sqrt2+\eta/2)\sqrt{\log N},
 \qquad
 \sum_{k=1}^Mf_k+a+M<(2\sqrt2+\eta)\sqrt{\log N}.
$$

For the moment parameters, set
$$
 \widehat q_k:=\frac{\sqrt2(1-\rho)\sqrt{(1-\varepsilon)\log N}}{\sqrt M\lambda_k},
 \qquad q_k:=\lfloor\widehat q_k\rfloor.
$$
{In particular, all $q_k$ are positive integers for $N$ large enough, due to the presence of the $\log N$ factor.}
To check admissibility, {that is, the condition \cite[(C.2)]{CNZ-max},} note that, for $N$ sufficiently large,
$\widehat\beta_{N,\varepsilon}^2/(M(1-\widehat\beta_{N,\varepsilon}^2k/M))\leq2/A$. The inequality $\log(1+t)\geq t/(1+t)$ yields
\begin{align*}
   \frac{1}{(1-\varepsilon)\log N}\binom{q_k}{2}
 &\leq\frac{(1-\rho)^2}{M\lambda_k^2}\leq(1-\rho)^2\left(\frac{1-\widehat\beta_{N,\varepsilon}^2k/M}{\widehat\beta_{N,\varepsilon}^2}+\frac1M\right)\\
& \leq(1-\rho)^2(1+2/A)\frac{1-\widehat\beta_{N,\varepsilon}^2k/M}{\widehat\beta_{N,\varepsilon}^2}.
\end{align*}
Consequently, {the analogue of \cite[(C.2)]{CNZ-max} holds, since}
$$
 \limsup_{N\to\infty}\frac{1}{(1-\varepsilon)\log N}\binom{q_k}{2}
 <\frac{1-(1-\varepsilon)k/M}{1-\varepsilon},
 \qquad 1\leq k\leq M.
$$

Set $r_{-1}:=r_0:=1$ and $r_j:=\lceil n^{j/(2M)}\rceil$ for
$1\leq j\leq M$, and define the block maxima
$$
 \overline W_{k,N}:=
 \max_{x\in[0,r_{k-2})^2\cap\mathbb Z^2}W_{k,N}(x).
$$
By \cite[Lemma~3.7]{CNZ-max}, for nonnegative 
$\alpha_k,\ldots,\alpha_M$,
\begin{equation}\label{eq:CNZ-prob-bound}
\begin{aligned}
 &\mathbb P\left(
   \log_+\overline W_{\ell,N}\geq\alpha_\ell
   \quad\forall\ell\in\{k,\ldots,M\}\right)\\
 &\qquad\leq n^{-2}+n^{2/M}
 \exp\left(-\sum_{\ell=k}^Mq_\ell\alpha_\ell
       +\sum_{\ell=k}^M\lambda_\ell^2\binom{q_\ell}{2}\right).
\end{aligned}
\end{equation}

For $1\leq k\leq M$, let $\mathcal B_k$ consist of the vectors
$(\alpha_k,\ldots,\alpha_M)$ with nonnegative integer entries satisfying
$$
 \begin{aligned}
 a+\sum_{\ell=k}^Mf_\ell
 &<\sum_{\ell=k}^M\alpha_\ell
 \leq a+\sum_{\ell=k}^Mf_\ell+1,\\
 \sum_{\ell=j}^M\alpha_\ell
 &\leq a+\sum_{\ell=j}^Mf_\ell,
 \qquad k<j\leq M.
 \end{aligned}
$$
If $\log Z_n(x)>\sum_{\ell=1}^Mf_\ell+a+M$, then
$\sum_{\ell=1}^M\lfloor\log_+W_{\ell,N}(x)\rfloor
>\sum_{\ell=1}^Mf_\ell+a$.
{Choose the largest index $k$ for which the inequality $\sum_{\ell=k}^M\lfloor\log_+W_{\ell,N}(x)\rfloor
>\sum_{\ell=k}^Mf_\ell+a$ holds.}
This produces 
$\alpha\in\mathcal B_k$ with  {$\lfloor \log_+W_{\ell,N}(x)\rfloor=\alpha_\ell$ for every $\ell>k$ and 
$\lfloor \log_+W_{k,N}(x)\rfloor\geq \alpha_k$, and in particular we have that
$\log_+W_{\ell,N}(x)\geq\alpha_\ell$ for every $\ell\geq k$.}
 {(See \cite[Proposition 3.1]{CNZ-max} for a similar argument.)} Therefore,
\begin{align}
&   \mathbb P\left(\max_{x\in {\rm B}(z,\sqrt{N^{1-\varepsilon}})}
       \log Z_n(x)>(2\sqrt2+\eta)\sqrt{\log N}\right)\label{eq: union bound for max subcritical}\\
       &\leq \sum_{k=1}^M\sum_{\alpha\in\mathcal B_k} \mathbb P\left(\exists x\in {\rm B}(z,\sqrt{N^{1-\varepsilon}}):\ \log_+W_{\ell,N}(x)\geq\alpha_\ell\ \forall\ell\in\{k,\ldots,M\}\right).\nonumber
       \end{align}

Moreover, for a constant $C=C(\eta,\varepsilon,M)$,
$$
 \sup_{\alpha\in\mathcal B_k}\sum_{\ell=k}^M\alpha_\ell
 \leq C\sqrt{(1-\varepsilon)\log N},\qquad
 |\mathcal B_k|\leq(C\sqrt{(1-\varepsilon)\log N}+2)^M.
$$

Since $f_k$ is increasing, \cite[Lemma~2.1]{CNZ-max} gives
$$
 \sum_{\ell=k}^M\frac{\alpha_\ell}{f_\ell}
 \geq M-k+1+\frac{a}{f_M},
 \qquad \alpha\in\mathcal B_k.
$$
Multiplying by $2(1-\rho)(1+\rho)(1-\varepsilon)\log N/M$ and allowing for integer
rounding, we obtain
$$
 \sum_{\ell=k}^Mq_\ell\alpha_\ell
 \geq\frac{2(1-\rho)(1+\rho)(M-k+1)}{M}(1-\varepsilon)\log N
       +\widehat q_Ma-C\sqrt{(1-\varepsilon)\log N}.
$$
On the other hand,
$$
 \sum_{\ell=k}^M\lambda_\ell^2\binom{q_\ell}{2}
 \leq\frac{(1-\rho)^2(M-k+1)}{M}(1-\varepsilon)\log N.
$$
Since $2(1-\rho)(1+\rho)-(1-\rho)^2=1+2\rho-3\rho^2>1$,
the exponent in \eqref{eq:CNZ-prob-bound} is at most
$$
 -\frac{M-k+1}{M}(1-\varepsilon)\log N-\widehat q_Ma+C\sqrt{(1-\varepsilon)\log N}.
$$

We cover ${\rm B}(z,\sqrt{N^{1-\varepsilon}})$ by at most
$Cn/r_{k-2}^2\leq Cn^{(M-k+2)/M}$ translates of
$[0,r_{k-2})^2\cap\mathbb Z^2$. Thus, \eqref{eq: union bound for max subcritical} and \eqref{eq:CNZ-prob-bound}  give
$$
 \begin{aligned}
 &\mathbb P\left(\max_{x\in {\rm B}(z,\sqrt{N^{1-\varepsilon}})}
       \log Z_n(x)>(2\sqrt2+\eta)\sqrt{\log N}\right)\\
 &\quad\leq CM(C\sqrt{(1-\varepsilon)\log N}+2)^M \\
 &\qquad \times
 \left[n^{-1+1/M}
       +\exp\left(\frac{3(1-\varepsilon)\log N}{M}-\widehat q_Ma+C\sqrt{(1-\varepsilon)\log N}\right)\right].
 \end{aligned}
$$
Finally, $\lambda_M^2=\log(1+\widehat\beta_{N,\varepsilon}^2/(M(1-\widehat\beta_{N,\varepsilon}^2)))\leq \widehat\beta_{N,\varepsilon}^2/(M(1-\widehat\beta_{N,\varepsilon}^2))$,
so, for $N$ sufficiently large,
$$
 \widehat q_Ma
 \geq\frac\eta4\sqrt2(1-\rho)
          \sqrt{\frac{1-\widehat\beta_{N,\varepsilon}^2}{\widehat\beta_{N,\varepsilon}^2}}\,(1-\varepsilon)\log N
 \geq\frac{\eta(1-\rho)}4\sqrt\varepsilon\,(1-\varepsilon)\log N.
$$
Choose $\varepsilon_0<1/2$ so small that
$3\varepsilon/A\leq\eta(1-\rho)\sqrt\varepsilon/8$ for
$0<\varepsilon<\varepsilon_0$.
  Using $1-\varepsilon>1/2$, we obtain
\eqref{eq:CNZ-tail} with some $c>0$ independent of
$\varepsilon$. 
\end{proof}
\section{A stopping-time estimate}
We bound the factor contributed by a visit to the origin at a prescribed
dyadic scale. This is the estimate used in Lemma~\ref{lem:step3}.

Let $S$ and $\tilde S$ be two independent simple random walks on
$\mathbb Z^2$. We write ${\mathrm E} _0$ for expectation when both walks
start at the origin and $p_m(z):=\mathrm P_0(S_m=z)$. 

\begin{lemma}\label{lem:weighted-first-hit}
For every fixed $B>0$, there exists $C=C(B)>0$ such that,
for all sufficiently large $N$ and all integers $1\leq v\leq N$,
\begin{equation}
\label{eq-141026c}
{\mathrm E} _0\left[
\exp\left(\beta_N^2\sum_{n=1}^v\mathbf1_{\{S_n=\tilde S_n\}}\right)
\mathbf1_{\{\exists n\in(v,v+Bv]:S_n=\tilde S_n\}}\right]
\leq C\left(\frac1{\log(v+1)}
+\frac1{1+\log(N/v)}\right).
\end{equation}
In particular, for every $ {A\in [1,(\log N)/2]}$,
 uniformly for $e^A\leq v\leq Ne^{-A}$,
the left-hand side is bounded  {by $2C/A$.}
\end{lemma}
\begin{proof}
 {The seecond claim follows from \eqref{eq-141026c} by optimizing over $v$ in the indicated range, noting that the maximizer is achieved on the boundaries of the domain.}

 {Turning to the proof of \eqref{eq-141026c},}
set
\[
D_n:=S_n-\widetilde S_n,
\qquad
q_n:=\mathrm P_0(D_n=0)=p_{2n}(0),
\]
and recall that
\[
\sigma_N^2=e^{\beta_N^2}-1.
\]
For \(n\ge1\), let
\[
U(n):=\sum_{r\ge1}(\sigma_N^2)^r q^{*r}(n),
\qquad U(0):=1.
\]
 {where $q^{*r}$ is the $r$-fold convolution of $q$.}

We first recall a renewal estimate from
\cite[Theorem~1.4 and
Proposition~1.6]{CSZ-Dickman}. Using \eqref{eq-091026b}
we have
\[
\lambda_N:=\sigma_N^2R_N
=1+O\left(\frac1{\log N}\right).
\]
The weighted renewal function  {$U_{N,\lambda_N}(n)$} in
\cite[Theorem~1.4]{CSZ-Dickman} coincides, for \(n\le N\),
with \(U(n)\).  {(Note that \cite{CSZ-Dickman} work with $q_n/R_N$, which explains the multiplication by $R_N$.)} 
\corO{From 
\cite[Proposition~1.6]{CSZ-Dickman}\footnote{\corO{We need here to increase the parameter $\vartheta$ in that theorem and use
monotonicity in \(\lambda_N\), in order to match notation and work with a fixed $\vartheta$, see the footnote to \eqref{eq-091026a}}} we obtain}
\begin{equation}\label{eq:CSZ-renewal-bound}
 U(n)\le
 C\frac{\log N}
 {n(1+\log(N/n))^2},
 \qquad 1\le n\le N.
\end{equation}
In particular, for \(1\le m\le N\),
\begin{equation}\label{eq:CSZ-renewal-sum}
 \sum_{n=1}^m U(n)
 \le
 C\frac{\log N}{1+\log(N/m)}.
\end{equation}

We next estimate the probability of a collision after a time gap.
For \(0\le a\le v\)  {and $B>0$,
 we claim that for some $C$ depending on $B$ only,}
\begin{equation}\label{eq:future-collision}
  {F_v(a):=}\mathrm P_0\big(
\exists n\in(a,a+Bv],\,D_n=0
\big) \le
 \frac{C}{\log(v+1)}
 \left(
 1+\log\frac{v}{a\vee1}
 \right).
\end{equation}
Indeed, by the strong Markov property at
\[
\tau_a:=\inf\{n>a:D_n=0\},
\]
we have
\[
\begin{aligned}
F_v(a)\sum_{j=0}^{\lfloor Bv\rfloor}q_j
&=
{\mathrm E} _0\left[
\mathbf 1_{\{\tau_a\le a+Bv\}}
\sum_{j=0}^{\lfloor Bv\rfloor}
\mathbf 1_{\{D_{\tau_a+j}=0\}}
\right].
\end{aligned}
\]
On the event \(\{\tau_a\le a+Bv\}\), for every
\(0\le j\le \lfloor Bv\rfloor\),  {we have that
\(
\tau_a+j\le a+2Bv
\).}
Hence
\[
\mathbf 1_{\{\tau_a\le a+Bv\}}
\sum_{j=0}^{\lfloor Bv\rfloor}
\mathbf 1_{\{D_{\tau_a+j}=0\}}
\le
\sum_{m=a+1}^{\lfloor a+2Bv\rfloor}
\mathbf 1_{\{D_m=0\}}.
\]
Taking expectations gives
\[
F_v(a)\sum_{j=0}^{\lfloor Bv\rfloor}q_j
\le
\sum_{m=a+1}^{\lfloor a+2Bv\rfloor}q_m.
\]
Since \(q_j\asymp 1/j\) and
\(\sum_{j\le m}q_j\asymp  \log(m+1)\), this gives
\eqref{eq:future-collision}.

We now expand
\[
e^{\beta_N^2\sum_{j=1}^v\mathbf1_{\{D_j=0\}}}
=
\prod_{j=1}^v
\left(1+\sigma_N^2\mathbf1_{\{D_j=0\}}\right).
\]
Grouping terms according to the last selected collision time and
using the Markov property gives the  identity
\begin{equation}\label{eq: identity for D F}
    \begin{aligned}
{\mathrm E} _0\left[
e^{\beta_N^2\sum_{j=1}^v\mathbf1_{\{D_j=0\}}}
\mathbf1_{\{\exists n\in(v,v+Bv]:D_n=0\}}
\right] =
\sum_{n=0}^v U(n)F_v(v-n).
\end{aligned}
\end{equation}

By \eqref{eq:future-collision}, this is at most
\begin{equation}\label{eq:main-convolution}
\frac{C}{\log(v+1)}
\sum_{n=0}^v
U(n)
\left(
1+\log\frac{v}{(v-n)\vee1}
\right).
\end{equation}

The term for  \(n=0\) in \eqref{eq:main-convolution} is bounded by
\(C/\log(v+1)\). For \(1\le n\le v/2\), the logarithmic factor in
\eqref{eq:main-convolution} is bounded by a constant, and
\eqref{eq:CSZ-renewal-sum} gives
\[
\sum_{n\le v/2}U(n)
\le C\frac{\log N}{1+\log(N/v)}.
\]
Hence this part contributes \corO{to the sum in \eqref{eq:main-convolution}} at most
\[
C\frac{\log N}{\log(v+1)(1+\log(N/v))}
\le
C\left(
\frac1{\log(v+1)}+\frac1{1+\log(N/v)}
\right).
\]

For \(v/2<n\le v\), \eqref{eq:CSZ-renewal-bound} implies 
\(
U(n)\le C{\log N}/({v(1+\log(N/v))^2}).
\)   
Since
\[
\sum_{a=0}^{\lfloor v/2\rfloor}
\left(1+\log\frac{v}{a\vee1}\right)
\le Cv,
\]
\corO{making the change of variables $a=v-n$,} the contribution \corO{to the sum in \eqref{eq:main-convolution}} is bounded by
\[
C\frac{\log N}{\log(v+1)(1+\log(N/v))^2}
\le
C\left(
\frac1{\log(v+1)}+\frac1{1+\log(N/v)}
\right).
\]
Combining these bounds with \eqref{eq: identity for D F}, we obtain \eqref{eq-141026c}. 
\end{proof}
\section{A swapping identity}
\label{sec-swap}

\begin{lemma}\label{lem: swapping lemma}
Let $m\ge 0$ and let $\mathcal A$ be a finite set. For each $q\in\{0,\dots,m\}$, the swapping operator
$\tau_q$ is a bijection of $\mathcal A$. Assign a set
$I_U\subseteq\{0,\dots,m\}$ to each $U\in\mathcal A$, and suppose that
\[
I_{\tau_qU}=I_U
\qquad
\forall q\in\{0,\dots,m\},\ U\in\mathcal A.
\]
Let $A,B,F_0,\ldots,F_m:\mathcal A\to\mathbb R$. Suppose that, for
every $q\in\{0,\dots,m\}$ and $U\in\mathcal A$,
\[
A(\tau_qU)-B(\tau_qU)
=-\bigl(A(U)-B(U)\bigr)
\]
and whenever $\ell\ne q$, 
\[
F_\ell(\tau_qU)=F_\ell(U).
\]
Let $K_0,\ldots,K_m:\mathcal A\to\mathbb R$, and suppose that, for
every $q\in\{0,\dots,m\}$, $U\in\mathcal A$, and $\ell\le q$,
\[
K_\ell(\tau_qU)=K_\ell(U).
\]
Define
\[
\Gamma_\ell(U):=F_\ell(U)-K_\ell(U).
\]
Then
\begin{align*}
&\sum_{U\in\mathcal A}A(U)\prod_{\ell=0}^mF_\ell(U)
-\sum_{U\in\mathcal A}B(U)\prod_{\ell=0}^mF_\ell(U)\\
&\qquad=
\sum_{U\in\mathcal A}A(U)
\prod_{\ell\in I_U}\Gamma_\ell(U)
\prod_{\ell\in\{0,\dots,m\}\setminus I_U}F_\ell(U) \\
&\qquad \qquad -
\sum_{U\in\mathcal A}B(U)
\prod_{\ell\in I_U}\Gamma_\ell(U)
\prod_{\ell\in\{0,\dots,m\}\setminus I_U}F_\ell(U).
\end{align*}
\end{lemma}
\begin{proof}
For every $U\in\mathcal A$,
\begin{align*}
&\prod_{\ell=0}^mF_\ell(U)
-
\prod_{\ell\in I_U}\bigl(F_\ell(U)-K_\ell(U)\bigr)
\prod_{\ell\in\{0,\dots,m\}\setminus I_U}F_\ell(U)\\
&\qquad=
\sum_{\emptyset\ne S\subseteq I_U}
(-1)^{|S|+1}
\prod_{\ell\in S}K_\ell(U)
\prod_{\ell\in\{0,\dots,m\}\setminus S}F_\ell(U).
\end{align*}
For each nonempty $S$, we show that summing the corresponding term
against $A(U)-B(U)$ gives zero.

Fix a nonempty $S\subseteq\{0,\dots,m\}$, set $q:=\max S$, and define
\[
\mathcal A_S:=\{U\in\mathcal A:S\subseteq I_U\}.
\]
The identity $I_{\tau_qU}=I_U$ shows that $\tau_q$ restricts to a
bijection of $\mathcal A_S$. Since $q\in S$, every factor
$F_\ell$ in
\[
W_S(U):=
\prod_{\ell\in S}K_\ell(U)
\prod_{\ell\in\{0,\dots,m\}\setminus S}F_\ell(U)
\]
has $\ell\ne q$. Also, every $\ell\in S$ satisfies $\ell\le q$.
Consequently,
\[
W_S(\tau_qU)=W_S(U).
\]
Using the bijection $U\mapsto\tau_qU$ on $\mathcal A_S$,
\begin{align*}
\sum_{U\in\mathcal A_S}(A(U)-B(U))W_S(U)
&=
\sum_{U\in\mathcal A_S}
\bigl(A(\tau_qU)-B(\tau_qU)\bigr)W_S(\tau_qU)\\
&=
-\sum_{U\in\mathcal A_S}(A(U)-B(U))W_S(U).
\end{align*}
Hence, this sum is zero. Summing over all nonempty $S\subseteq\{0,\dots,m\}$ 
proves the identity.
\end{proof}

\end{appendices}

\section*{Acknowledgments}
We thank Stefan Junk for explaining the proof of \cite{Junk25} to the second author. S.N was supported by JSPS KAKENHI
Grant Numbers JP24K16937 and JP25K00911. O.Z. was partially supported by the Israel Science Foundation grant \#621/24.

\section*{Declaration on the use of AI}
We originally developed all mathematical arguments needed to prove our main results without the use of AI tools. In particular, we proved a crude lower-tail estimate for $\log Z_N$ {(a crude version of Proposition \ref{prop:iterated-lower-tail})}
using an inductive argument together with Talagrand's concentration inequality; that version was   sufficient for our purposes. During subsequent refinement of the proofs with AI assistance, ChatGPT Astra suggested that we use \cite{Chen,LX-moments}, where a stronger lower-tail estimate is available. {The present version of Proposition \ref{prop:iterated-lower-tail} uses the latter.}  We also use ChatGPT to check the mathematical arguments  and generate Figure~\ref{fig: fig1}.

\end{document}